\documentclass[a4paper,12pt]{amsart} 
\usepackage{amsmath,amssymb,amscd,amsfonts,amsthm,amsrefs}
\usepackage[colorlinks,linkcolor=blue,anchorcolor=blue,citecolor=blue]{hyperref}
\usepackage{enumerate}   
\usepackage{graphicx}
\numberwithin{equation}{section}
\newcommand{\p}{\partial}
\newcommand{\vphi}{\varphi}
\newcommand{\om}{\omega}
\newcommand{\tri}{\triangle}
\newcommand{\eps}{\epsilon}
\newcommand{\thmref}[1]{Theorem~\ref{#1}}

\newcommand{\lemref}[1]{Lemma~\ref{#1}}
\newcommand{\corref}[1]{Corollary~\ref{#1}}
\newcommand{\propref}[1]{Proposition~\ref{#1}}  

\newcommand{\Om}{\Omega}

\def\i{{\rm i}} 
\def\j{{\rm j}} 
\def\p{\partial}
\def\b{\beta}

\def\vol{{\rm vol}} 
\def\Vol{{\rm Vol}} 

\def\Diff{{\rm Diff}}

\def\det{\mathop{\rm det}\nolimits}

\def\a{\alpha}
\DeclareMathOperator{\Tr}{Tr}

\newtheorem{theorem}{Theorem}[section]
\newtheorem{thm}[theorem]{Theorem}
\newtheorem{rem}[theorem]{Remark}
\newtheorem{defn}[theorem]{Definition}

\newtheorem{lem}[theorem]{Lemma}

\newtheorem{prop}[theorem]{Proposition}

\newtheorem{cor}[theorem]{Corollary}

\allowdisplaybreaks[4]  

\title[Formal Riemannian structure]{The space of closed G2-structures. II. concavity and hypersymplectic structures}

\author{Kai Zheng}  
\address{University of Chinese Academy of Sciences, Beijing 100190, P.R. China}
\email{KaiZheng@amss.ac.cn}

\begin{document}
	\maketitle

\begin{abstract}
		In the space of closed $G_2$-structures equipped with Bryant's Dirichlet-type metric, we continue to utilise the geodesic, constructed in our previous article, to show that, under a normalisation condition Hitchin's volume functional is geodesically concave and the $G_2$ Laplacian flow decreases the length. Furthermore, we also construct various examples of hyper-symplectic manifolds on geodesic concavity. 
	\end{abstract}


\section{Introduction}
Let $M$ be a closed $7$-manifold and $[\vphi]$ be a given cohomology class in $H^3(M,\mathbb R)$ containing at least one $G_2$-structure. Torsion-free $G_2$-structures are local maximums of Hitchin's volume functional \cite{MR1863733}.

We denote by $\mathcal{M}$ the set of all closed $G_2$-structures in $[\vphi]$ and by $\Diff_0$ the identity component of the diffeomorphism group. A fundamental question in $G_2$ geometry is the uniqueness question: whether the torsion-free $G_2$-structure is unique up to $\Diff_0$ in each connected component of $\mathcal{M}$.
This question was brought to prominence in Joyce \cite{MR1787733}, Donaldson \cite{MR3966735,MR3932259,MR3702382}.

Another motivation arises from studying the fundamental uniqueness question in the space of symplectic forms.
Two symplectic forms $\om_0$ and $\om_1$ are called
 isotopic, if they are connected by a path of cohomologous symplectic forms; called
diffeomorphic, if there exists a diffeomorphism $\sigma$ such that $\sigma^\ast\om_0=\om_1$.

	A hyper-symplectic structure on a $4$-dimensional manifold $X^4$ refers to a triple of symplectic $2$-forms $\underline\omega:=(\omega_1,\omega_2,\omega_3)$, which spans a maximal positive subspace for the wedge product.
	Donaldson \cite{MR2313334} suggest investigating the following question:
whether $\underline\om$ is diffeomorphic to a hyper-K\"ahler structure $\underline\om^K$?
Thanks to Moser's theorem, it is sufficient to show that they are isotopic.
We direct the reader to \cite{MR3992720,MR3098204} for further recent development.

These questions are geometrically intertwined.
After adding a $\mathbb T^3$, we have a product 7-manifold $M:=\mathbb{T}^3\times X^4$. The hyper-symplectic structure $\underline\om$ determines a closed $G_2$ structure on $M$
	\begin{align}\label{eq:G2}
		\vphi:= d\theta^{123}-d\theta^i\wedge \omega_i.
	\end{align}
When $M$ is closed, the torsion freeness is equivalent to hyper-K\"ahler.
The hyper-symplectic flow \cite{MR3881202}, which is the restriction of the $G_2$ flow on $X^4$, is used to construct the cohomologous path. We refer to \cite{MR3873112,arXiv:2404.15016,MR4085670} for examples and recent progress.

Following the latest progress in K\"ahler geometry, the geodesic approach provides a natural framework for investigating these uniqueness questions among canonical structures.

We equip the infinite dimensional space $\mathcal{M}$ with a Dirichlet-type metric $\mathcal G$ \cite{MR2282011}.
In \cite{MR4732956}, we start to explore the global geometry of the infinite dimensional Riemannian manifold $(\mathcal{M},\mathcal G)$, computing the Levi-Civita connection for $\mathcal G$ and formulating the corresponding geodesic equation, which will be called \textsl{canonical connection} and \textsl{canonical geodesic} in this article. 

In this article, building on previous results, we further develop the global variational analysis of the volume functional, with particular focus on establishing its geodesic concavity and length contraction along gradient flow trajectories.
	
	We recall the formula of the Dirichlet-type metric $\mathcal G$ on $\mathcal{M}$.
	The tangent space of $\mathcal{M}$ at an element $\vphi$ is identified to $d\Om^2$.
	The expression of $\mathcal G$ is 
	\begin{align}\label{metric}
		\mathcal G_\vphi(X,Y)=\int_Mg_\vphi(\delta  u,\delta  v) \cdot \vol_\vphi\text{, for all } X,Y\in T_\vphi\mathcal{M},
	\end{align}
	where, $ u$, $ v$ are two $d$-exact $3$-forms satisfying the Hodge Laplacian equation $\tri u=X$, $\tri v=Y$, with respect to the metric $g_\vphi$ determined by the $G_2$-structure $\vphi$.


In Section~\ref{HS and G2}, we extend the metric~\eqref{metric} to the space of $2$-forms and derive the associated geodesic equation (Theorem~\ref{2 form Levi Civita} and Definition~\ref{$L^2$ geodesic}). Our construction adapts the strategies of~\cite{MR4732956} to this setting. 
Having made these preparations, we are able to derive the Hessian of the volume functional by differentiating twice along geodesics in $(\mathcal{M},\mathcal G)$. The computations in this section are taken in a convectional way, that leads to the natural issues of the solvability and regularity of the geodesics in the future study.

Additionally, we introduce a normalisation condition $\mathcal{N}$ (Definitions~\ref{condition N} and~\ref{N geodesics}) from calculation. It appears that the Ricci curvature and the scale curvature appear in
the expression of $\mathcal{N}$ and the Hessian as well.

\begin{thm}In the infinite-dimensional space $(\mathcal{M},\mathcal G,\Vol)$,
\begin{enumerate}
\item the volume functional is concave, along the normalised geodesic;
\item the gradient flow of the volume functional decreases the length of all the normalised paths.  
\end{enumerate}
\end{thm}
Complete statements with rigorous formulations are presented in \thmref{Hessian of Vol general} for concavity and \thmref{length contraction} for length contraction.


In Section \ref{Hypersymplectic structures}, we move to hypersymplectic manifolds, formulating the tangent space of $\mathcal{M}$, the metrics on $\mathcal{M}$ and the volume functional.


In Section~\ref{Mixed structures}, we specialise to the case where two components of the hyper-symplectic triple are fixed while the third varies. For this setting, we derive explicit formulas for:
 the induced metric,
 the associated connection,
the geodesic equations, and
the Hessian of the volume functional,
which will be useful for subsequent constructions of examples.

In Section~\ref{Torus fibration}, we utilise the calculation in Section~\ref{Mixed structures} to show explicit examples satisfying geodesic concavity. We further assume that:
the underlying manifold $X^4$ is a 4-torus,
and the hyper-symplectic structure $\underline{\omega}$ is homogeneous of degree one.
We then define the space $\mathcal{M}_1$ (Definition~\ref{defn:space 1}) as the collection of all $G_2$ structures induced by such hyper-symplectic structures within a fixed equivalence class.

\begin{thm}\label{theorem main}
In the space $(\mathcal M_1,\mathcal G, \Vol)$,
\begin{enumerate}
\item the volume is geodesically concave along the canonical geodesic, Proposition \ref{geodesic concavity body};
\item the gradient flow of the volume decreases the distance, Proposition \ref{distance HS flow};
\item the sectional curvature is nonnegative, Proposition \ref{Curvature body}.
\end{enumerate}
\end{thm}
	
In this example,
the geodesic concavity of the volume functional yields the uniqueness of the torsion-free $G_2$ structure modulo diffeomorphisms (Theorem~\ref{uniqueness body}), offering a positive evidence to the uniqueness questions posed earlier. Furthermore, Proposition~\ref{diffeomorphismsgeodesic} shows that the diffeomorphism solves the canonical geodesic equation. The geodesic equation takes the form of a Hamilton-Jacobi equation as presented in \lemref{equations geodesic}.

In the proof of \thmref{theorem main}, we apply the $G_2$ formalism developed in the previous sections. Actually, in this example, \thmref{theorem main} revisits the Otto calculus of the geometry of the infinite-dimensional Riemannian manifold $(\mathcal{P}_2, W_2)$ from optimal transport theory \cite{MR1842429,MR2459454}, as discussed in Remark~\ref{optimal transport}. By interpreting the symplectic form as a measure in the $2$-Wasserstein space $\mathcal{P}_2$, we observe that the Dirichlet metric coincides with the Wasserstein metric on velocity fields, thereby inducing the Wasserstein distance $W_2$.

In Section~\ref{concavity}, we extend Theorem~\ref{theorem main} to triples of structures from $\mathcal{M}_1$. Specifically, we introduce the space $\mathcal{M}_3$ (Definition~\ref{defn:space 3}) consisting of all triples of elements from $\mathcal{M}_1$, and prove that the geometric structure $(\mathcal{M}_3, \underline{\mathcal{G}}, \Vol_{\underline{\chi}})$ satisfies properties analogous to those in Theorem~\ref{theorem main}, including the case of weighted volume functionals. The complete formulation appears in Theorem~\ref{M3 structure}.

\begin{rem}
The Dirichlet-type metric on the space of closed $G_2$ structures admits a natural comparison with its counterpart in the space of K\"ahler metrics in \cite{MR3412344}.
\end{rem}


In Section~\ref{non-concavity}, we present an example, demonstrating the failure of geodesic concavity for the weighted volume functional along canonical geodesics. 
We set the space $\widetilde{\mathcal{M}}_3$ of closed $G_2$ structures admitting the form~\eqref{eq:G2} within a fixed cohomology class on $\mathbb{T}^7$. 

\begin{thm}[\thmref{general weight}]
For the space $(\widetilde{\mathcal{M}}_3, \mathcal{G})$, the volume functional fails to be concave along the canonical geodesics. Moreover, this non-concavity persists for all weighted volume functionals.
\end{thm}


	\noindent {\bf Acknowledgments:}	
The author is partially funded by NSFC grants No. 12171365.
He would like to acknowledge support from the CRM and the Simons Foundation through the Simons CRM scholar program, and the ICTP through the Associates Programme (2020-2025).


\section{$G_2$ structures}\label{HS and G2}

\subsection{Riemannian structure for $\Om^2$}
We will present the Riemannian structure for the space of $2$-forms, including the metric, the Levi-Civita connection and the geodesic equation, extending our previous work \cite{MR4732956}.
\begin{defn}[$L^2$ metric for $2$-forms]
We define the metric
	\begin{align}\label{L2 metric}
		\mathcal G_\vphi(\a,\b):=\int_Mg_\vphi(\alpha,\beta) \cdot \vol_\vphi,\quad \forall \alpha,\beta\in \Om^2.
	\end{align}
\end{defn}
\begin{rem}
When $\a,\b$ are chosen to be the canonical forms $\delta u,\delta v$, respectively, the $L^2$ metric \eqref{L2 metric} coincides with the Dirichlet metric \eqref{metric}. Any $2$-form $\alpha$ and the canonical form differ by a $d$-closed form. 
\end{rem}

\begin{defn}\label{L2 connection}
Let $\vphi(t)$ be a family of closed $G_2$ structures in the space $\mathcal M$ of the closed $G_2$ structures determined by a given class $[\vphi]$. 

We assume that $\alpha(t)$ is a family of $2$-forms satisfying $\p_t\vphi=d\alpha$. We also attach another vector field $d\beta(t)$ along $\vphi(t)$.

We define the covariant derivative $D$ of $\beta$ along the family $\vphi(t)$: $$D_\a\b:=\beta_t+P(\vphi, \a,\b),$$ where $P$ is given by
\begin{equation}\label{connection P}
\begin{split}
		P(\vphi,\a,\b):=&\frac{1}{2}\big[g_\vphi(\vphi,d\a)\beta+g_\vphi(\vphi,d\b)\alpha\big]\\
		&-\frac{1}{2}\delta\big[g_\vphi(\alpha,\beta)\vphi-\i_\vphi \j_{\alpha}\beta\big]
		-\frac{1}{4}\big[\i_{\beta}\j_\vphi d\a+\i_{\alpha}\j_\vphi d\b\big].
\end{split}
\end{equation}
Here the operators $\i,\j$ for a $p$-form $\omega$ were defined in \cite[Section 2.3]{MR4732956}:
\begin{align*}
&\i_\omega: S^2(T^*M)\rightarrow\Omega^p(M),
\quad \j_\omega: \Omega^p(M) \rightarrow S^2(T^*M),\\
& \i_{\omega}h=\frac{1}{(p-1)!} \, h_{i_1l} \, g^{ls} \, \omega_{si_2\cdots i_p} \, dx^{i_1\cdots i_p},\\
	&\j_{\omega}\omega_2=\frac{1}{2}[\omega_{ia_2\cdots a_p}(\omega_2)_j^{\ a_2\cdots a_p} +\omega_{ja_2\cdots a_p}(\omega_2)_i^{\ a_2\cdots a_p}]dx^i\otimes dx^j.
\end{align*}
When $p=3$, they are exactly Bryant's contraction operators \cite{MR2282011}.

\end{defn}
\begin{thm}\label{2 form Levi Civita}
The connection $D$ is a Levi-Civita connection, with respect to the $L^2$ metric \eqref{L2 metric}.
\end{thm}
\begin{defn}\label{$L^2$ geodesic}
We say a family of $\alpha(t)$ defined in Definition \ref{2 form Levi Civita} is an 
$L^2$ geodesic if it solves the nonlinear equation 
\begin{align}\label{alpha geodesic}
0=\alpha_t+g(\vphi,d\alpha)\alpha
	-\frac{1}{2}\delta(|\alpha|^2\vphi-\i_\vphi \j_{\alpha}\alpha)
	-\frac{1}{2}\i_{\alpha}\j_\vphi d\alpha.
\end{align}

We say a family of $3$-forms $\vphi(t)\in \mathcal M$ is a geodesic, if there exists an $L^2$ geodesic $\alpha(t)$ such that $\p_t\vphi(t)=d\alpha(t)$. 
\end{defn}

We collect some identities, which will be used later.
\begin{lem}\label{i j switch}
The $3$-form $d\alpha$ is decomposed into three parts 
\begin{align*}
&X_1:=\pi_1^3d\alpha=\frac{1}{7}g(\vphi,d\alpha)\vphi,\quad X_7:=\pi_7^3d\alpha=\frac{1}{4} * [\ast(d\alpha\wedge\vphi)\wedge \vphi],
	\end{align*}
and $X_{27}:=\pi_{27}^3d\alpha$, see \cite[Equation (3.1)]{MR4732956}.

The contraction operators satisfy the identities \cite[Lemma 2.18]{MR4732956}
\begin{align*}
&\int_Mg(d\gamma,\,\i_\vphi\j_\a\b)\vol_\vphi
	=\int_Mg(\j_\vphi d\gamma,\,\j_\a\b)\vol_\vphi
	=\frac{1}{2}\int_Mg(\i_\a\j_\vphi d\gamma,\,\b)\vol_\vphi,
	\end{align*}
	$\j_\a\b=\j_\b\a,\quad j_\vphi X_1=\frac{6}{7}g(\vphi,d\alpha) g,\quad j_\vphi X_7=0,\quad i_\vphi j_\vphi X_{27}=4 X_{27}$ and the norm
$
|j_\vphi X_1|^2=18 |X_1|^2 
$, see \cite[Proposition 2.9]{MR4732956}.

\end{lem}

\begin{proof}[Proof of \thmref{2 form Levi Civita}]
Let $\vphi(t,s)$ be a two-parameter family of closed $G_2$ structures such that $\p_s\vphi=d\beta$.
The torsion freeness following from computing $D_\a\b-D_\b\a=0$, since $d \beta_t=d \a_s=\vphi_{st}$ and $P(\vphi,\a,\b)=P(\vphi,\b,\a)$.

For the metric compatibility, we are given another vector field $Z=d\gamma$ attached along $\vphi(t)$. On the one hand, we
differentiate the metric $\mathcal G_\vphi(\beta,\gamma)$ in \eqref{L2 metric},
\begin{equation*}
\begin{split}
LHS&:=\p_t\mathcal G_\vphi(\beta,\gamma)-\mathcal G_\vphi(\beta_t,\gamma)-\mathcal G_\vphi(\beta,\gamma_t)
=\mathcal G_\vphi(\ast\ast_t\beta,\gamma).
\end{split}
\end{equation*}

On the other hand, we insert \eqref{connection P} into $\mathcal G_\vphi(P(\vphi,\alpha,\beta),\gamma)$ 
\begin{equation*}
\begin{split}
&
=\int_M [\frac{1}{2} g(\vphi,d\a)g( \beta,\gamma)
+\frac{1}{2} g(\vphi,d\b)g( \alpha,\gamma)
-\frac{1}{2}g(\vphi, d\gamma)g(\alpha,\beta)\\
		&+\frac{1}{2}g(\i_\vphi \j_{\alpha}\beta, d\gamma)
		-\frac{1}{4}g(\i_{\beta}\j_\vphi d\a,\gamma) -\frac{1}{2}g(d\b,\i_\vphi \j_{\alpha}\gamma) 
]\vol_\vphi.
\end{split}
\end{equation*} 
Here, we apply \lemref{i j switch} to the last term. Removing the anti-symmetric part, we have the expression of the $RHS$ defined to be
\begin{equation*}
\begin{split}
\mathcal G_\vphi(P(\vphi,\alpha,\beta),\gamma)
+\mathcal G_\vphi(P(\vphi,\alpha,\gamma),\beta)
=\int_M [ g(\vphi,d\a)g( \beta,\gamma)
-\frac{1}{2}g(\i_{\beta}\j_\vphi d\a,\gamma) 
]\vol_\vphi.
\end{split}
\end{equation*} 
Therefore, $LHS=RHS$, due to \cite[Proposition 3.3]{MR4732956}, which infers the metric compatibility.
\end{proof}

\subsection{Gauge fixing tangent space}
\begin{lem}\label{14 2 properties}
When $\alpha\in \Om_{14}^2$, we have $g(\vphi,d\a)=g(\tau,\a)$,
$$
 X_1:=\pi_1^3 d\alpha=\frac{1}{7}g(\tau,\a)\vphi, \quad 
 X_7:= \pi_7^3 d\alpha=\frac{1}{4}\ast(\vphi \wedge \delta\alpha).$$
Moreover, the norm $ |X_1|^2=\frac{1}{7}g^2(\tau,\a)$.
\end{lem} 
\begin{proof}
Following the proof of \cite[Proposition 2.19]{MR4732956}, we obtain the first two identities and $\delta\alpha=\ast(\vphi\wedge d\alpha)$ for $\alpha\in \Om_{14}^2$.
Then we plug $\delta\alpha$ into the expression $\pi_7^3 d\alpha
=\frac{1}{4}*[\vphi \wedge *(\vphi \wedge d\alpha)]$ and prove the lemma.
\end{proof}

According to the $G_2$-module decomposition, any $2$
-form $\alpha$ is decomposed into $\alpha=\alpha_7^2+\alpha_{14}^2$.
The $\Om_7^2$ part $\alpha_7^2=i_\mathcal V \vphi$ of the $2$-form $\alpha$ is the diffeomorphism action, where $\mathcal V$ is a vector field on $M$. It means there exists a family of diffeomorphisms $\sigma(s)$ such that
$$\p_s[\sigma^\ast(s) \vphi]=di_\mathcal V \vphi.$$

When we work on the gauge fixing space $\mathcal M_0:=\frac{\mathcal M}{\Diff_0}$ and suppose that $H^2=0$. In \cite{MR3959094}, the tangent space of $\mathcal M_0$ is a subspace of $\Om_{14}^2$,
\begin{align}\label{tangent space of gauge fixed space}
T\mathcal M_0=\frac{d\Om^2}{d\Om^2_7}=\frac{\Om^2_{14}}{\pi_{14}^2 d\Om^1}=\{\alpha\in\Om_{14}^2\vert \delta \a=0\}.
\end{align}

Note that the minimiser, denoted by $\chi$, of the $L^2$ norm of $2$-forms within the same class of $\alpha_{14}^2$,
$$\int_M (\alpha_{14}^2+d\eta)^2\vol, \quad \forall \eta\in\Om^1$$  satisfies the critical equation $\delta \chi=0$. We also have $d\alpha_{14}^2=d\chi$. But $\chi$ may not be $\Om_{14}^2$.

We will write $\pi^\bot$ the projection from $\Om^2$ to $T\mathcal M_0$ and $\pi^\parallel$ be the projection to its complementary $d\Om_7^2$. Then, we have $$\alpha^\bot:=\pi^\bot\alpha=\tilde \alpha_{14}^2,\quad \alpha^\parallel=\alpha_7^2,\quad 
d\alpha=d\alpha^\bot+d\alpha^\parallel.$$ We also write
$X^\bot=d\alpha^\bot, \, X^\parallel=d\alpha^\parallel.$ The decomposition is orthogonal under the $L^2$ metric \eqref{L2 metric} for $2$-forms.

\subsection{Variations and linearised operators}\label{Variations and linearised operators}
\begin{defn}\label{condition N}
We define a linear operators for $\alpha\in \Om^2$
\begin{equation*}
\begin{split}L (\alpha):= \delta[\frac{4}{3}X_1+X_7-X_{27}]= \delta[\frac{7}{3}X_1+ 2X_7- d\alpha].
\end{split}
\end{equation*}
Moreover, $L (\alpha)=\frac{7}{3}\delta X_1- \tri \alpha$, $\forall\alpha\in T\mathcal M_0$.
We also write
\begin{equation*}
\begin{split}
\mathcal L(\alpha,\beta):=\int_M g(L (\alpha),\beta) \vol_\vphi,
\quad \mathcal L(\alpha):=\frac{4}{3}\|X_1\|_L^2+\|X_7\|_L^2-\|X_{27}\|_L^2.
\end{split}
\end{equation*}
\end{defn}
The norm $\|X\|^2_L:=\mathcal G^L_\vphi(X,X)$ is measured by the Laplacian metric \cite[Definition 5.24]{MR4732956}
\begin{align*}
\mathcal G^L_\vphi(X,Y)=\int_Mg_\vphi(X,Y) \cdot \vol_\vphi,\quad \forall X,Y\in T_\vphi\mathcal M.
\end{align*}
\subsubsection{First variation}
Differentiate the volume along $\vphi_t=d\alpha$,
\begin{equation*}
\begin{split}
	3\Vol_t&=\int_M g( \vphi_t ,\vphi) \vol_\vphi
	=\int_M g(d\alpha,\vphi) \vol_\vphi
	=\int_M g( \alpha ,\tau) \vol_\vphi.
\end{split}
\end{equation*}
Since $\tau\in\Om_{14}^2$, we see that only the $\Om^2_{14}$ part $\alpha^2_{14}$ is involved. We use $\ast\tau=-d\psi$ and $\psi=\ast\vphi$ is the dual form. Then
\begin{equation}\label{eq:1st_vol general}
\begin{split}
	3\Vol_t
	=\int_M \alpha^2_{14} \wedge\ast\tau
	=-\int_M \alpha^2_{14} \wedge d\psi
	=\int_M d\alpha^2_{14} \wedge\psi.
\end{split}
\end{equation}

The torsion free structure satisfies the critical equation $$\tau=\delta\vphi=0.$$ From \cite[Corollary 3.8]{MR4732956}, the linearisation of $\tau$ is
\begin{align*}
\tau_t=L(\alpha)-g(\vphi,d\alpha)\tau+\frac{1}{2}\i_\tau \j_\vphi d\alpha=L(\alpha)-\frac{1}{7}g(\vphi,d\alpha)\tau+\frac{1}{2}\i_\tau \j_\vphi X_{27}.
\end{align*} 


\subsubsection{Second variation}
We further compute along $\vphi_s=d\beta$
\begin{equation*}
\begin{split}
	3\Vol _{ts}=\int_M \psi_s\wedge d\alpha^2_{14} .
\end{split}
\end{equation*}
Plug variation of $\psi$, $\psi_t=\ast[\frac{4}{3}X_1+X_7-X_{27}]$, see \cite[Corollary 3.5]{MR4732956}, we get the quadratic form
\begin{equation}\label{eq:2nd_vol general part 1} 
\begin{split}
3\Vol _{ts}
=\int_M g(L (\beta), \alpha_{14}^2) \vol_\vphi=\mathcal L(\beta,\alpha_{14}^2).
\end{split}
\end{equation}
\begin{lem} 
The linearised operator $L$ is self-adjoint, vanishes on $\Om_7^2$ and takes values in $\Om_{14}^2$, c.f. \cite[Proposition 2.2]{MR3959094}. 
\end{lem}
\begin{proof}
We give a direct proof as following. 
We use $\p_t\vol_\vphi=\frac{1}{3}g(\vphi,d\alpha)\vol_\vphi$ and
$\p_t g(\tau,\beta)=-\frac{1}{2}g(i_\tau j_\vphi X_{27},\beta)-\frac{4}{21}g(\vphi,d\alpha)g(\tau,\beta)$ from \cite[Lemma 3.1]{MR4732956}.
Since $\int_M g(\tau,\beta)\vol_\vphi=0, \forall\beta\in \Om_7^2$, taking $t$-derivative, we obtain
\begin{align*}
0=\int_M g(\tau_t,\beta)\vol_\vphi+\p_t g(\tau,\beta)\vol_\vphi+g(\tau,\beta)\frac{1}{3}g(\vphi,d\alpha)\vol_\vphi.
\end{align*}
Hence, 
$
L(\alpha)=\tau_t-\frac{1}{2}i_\tau j_\vphi X_{27}+\frac{1}{7}g(\vphi,d\alpha)\tau \in\Om_{14}^2.
$
\end{proof}

At the critical point, the linearisation of the critical point $\tau$ equals to the linearisation of the volume functional, $$\tau_t=L.$$ Furthermore, $L(\alpha)=- \tri \alpha,\quad \forall \alpha\in T\mathcal M_0$.
Local deformation has been constructed in \cite{MR1424428}.


\subsection{The geodesic equations modulo diffeomorphisms}\label{14 2 projection}
\begin{defn}\label{integration N}
We define a normalisation condition
\begin{equation*}
\begin{split}
&\mathcal N(\alpha)
:=\|X^\bot\|^2_L
+\mathcal L(\alpha)+\mathcal R(\a)
+\int_M[
-g(\tau,\alpha)g(\vphi,d\a) 
+g(\j_{\a}\tau,\, \j_\vphi d\a)]\vol_\vphi.
\end{split}
\end{equation*}
Here, $R(\a,\a):=\frac{1}{2}|\a|^2|\tau|^2
-\frac{1}{2}g(\j_\vphi d\tau,\,\j_\a\a)$ and $\mathcal R(\a)=\int_M R(\a,\a)\vol_\vphi$.
\end{defn}

\subsubsection{Ricci and scalar curvature}
We could bring the Ricci curvature $Ric$ and scalar curvature $S$ of the Riemannian metric $g_{\vphi(t)}$, determined by the $G_2$ structure $\vphi(t)$, enter the formula of $\mathcal N$.
\begin{lem}
The term $R(\a,\a)$ has the expression
\begin{equation*}
\begin{split}
&
R(\a,\a)=-3S\cdot |\a|^2
-g(\i_\a Ric,\,\a)
-\frac{1}{8}g(\i_\a\j_\vphi [*(\tau \wedge \tau)],\,\a).
\end{split}
\end{equation*}
\end{lem}
\begin{proof}
The formula is obtained from plugging the expressions of $Ric$ and $S$ for closed $G_2$ structure  \cite{MR2282011}:
\begin{align}\label{eq:Ricci}
	4Ric(g)=|\tau|^2 g +\j_\vphi  d\tau-\frac{1}{2}\j_\vphi [*(\tau \wedge \tau)],\quad
S(g)=-\frac{|\tau|^2}{2}.
\end{align}
Then we obtain that $$-\frac{1}{2}g(\j_\vphi d\tau,\,\j_\a\a)=-2g(Ric,\,\j_\a\a)
+\frac{|\tau|^2}{2}g( g,\,\j_\a\a)
-\frac{1}{4}g(\j_\vphi [*(\tau \wedge \tau)],\,\j_\a\a).$$
We also use \cite[Definition 2.3, Proposition 2.10]{MR4732956} to see
$$g( g,\,\j_\a\a)= 2g(\a,\a)=2|\a|^2,\quad g(Ric,\,\j_\a\a)=\frac{1}{2}g(\i_\a Ric,\a)$$ and $g(\j_\vphi [*(\tau \wedge \tau)],\,\j_\a\a)=\frac{1}{2}g(\i_\a\j_\vphi [*(\tau \wedge \tau)],\,\a)$. Thus, we have proved the lemma.
\end{proof}
\begin{rem}
Note that the last term in the formula of $R$ has a priori estimate in terms of the scalar curvature
  \cite[Equation (2.31)]{MR2282011}
		\begin{align*}
			4S\cdot g\leq j_\vphi[\ast_\vphi(\tau\wedge\tau)]\leq -\frac{4}{3}S\cdot g
		\end{align*}
		and $0\geq-\j_\a\a\geq -\frac{2}{3}|\a|^2$, for $\alpha\in \Om_{14}^2$, see \cite[Theorem 2.23]{MR4732956}.
\end{rem}

\begin{lem}\label{last term}
$g(\j_{\a}\tau,\, \j_\vphi d\a)
=\frac{12}{7}g(d\a,\vphi) g(\a, \, \tau)
+g(\j_{\a}\tau,\,\j_\vphi X_{27})
$.
\end{lem}
\begin{proof}
We combine $\j_\vphi d\alpha=\frac{6}{7}g(d\a,\vphi) g+\j_\vphi X_{27}$ with
$
g(\j_{\a}\tau,\, g)
=2 g(\a, \, \tau)
$ to prove the identity.
\end{proof}
\subsubsection{Normalised geodesics}
Given a tangent vector field $X=d\alpha$, the choice of $\alpha$ given in Definition \ref{L2 connection} is not unique.
We will pose a suitable condition on $\alpha$.
\begin{defn}\label{N geodesics}
An $L^2$ geodesic $\alpha(t)$ in Definition \ref{$L^2$ geodesic} is called 
\begin{itemize}
\item a canonical geodesic, if $\alpha(t)=\delta u(t)$, c.f. \eqref{metric};
\item an $\Om^2_{14}$ geodesic, if
$\alpha(t)\in \Om^2_{14}$;
\item a gauge fixing geodesic, if $\alpha(t)\in T\mathcal M_0$;
\item an $\mathcal N$ geodesic,
if $\alpha(t)$ satisfies the normalisation $\mathcal N(\alpha(t))=0$;

\item an $\mathcal N_{14}$ geodesic, if
$\alpha(t)\in \Om^2_{14}$ and $\alpha(t)$ is an $\mathcal N$ geodesic, too.
\end{itemize}
We say a family of $\vphi(t)\in \mathcal M$ is a canonical ($\Om^2_{14}$, gauge fixing, $\mathcal N$, $\mathcal N_{14}$) geodesic, if there exists an $L^2$ geodesic $\alpha(t)$, which is canonical ($\Om^2_{14}$, gauge fixing, $\mathcal N$, $\mathcal N_{14}$ respectively), such that $\p_t\vphi(t)=d\alpha(t)$. 

We say a family of $\vphi(t)\in \mathcal M$ is a weak $\Om^2_{14}$ ($\mathcal N_{14}$) geodesic, if there exists a family of $\alpha(t)$ such that $\p_t\vphi(t)=d\alpha(t)$ and the $\pi^2_{14}$ part $\pi_{14}^2\alpha(t)$ is an $\Om^2_{14}$ ($\mathcal N_{14}$ respectively) geodesic. 
\end{defn}

\begin{rem}
If $\alpha(t)\in \Om^2_7$, then $\alpha(t)$ is a weak $\Om^2_{14}$ geodesic, and also a weak $\mathcal N_{14}$ geodesic, by definition.
\end{rem}



\subsubsection{$\Om_{14}^2$ geodesic}
\begin{lem}\label{142 geodesic}
The $\Om_{14}^2$ geodesic $\vphi_t=d\alpha$ satisfies the equation
\begin{align*}
0=\alpha_t
+g(\tau,\alpha)\alpha
	-\frac{1}{2}\delta[|\alpha|^2\vphi-\i_\vphi \j_{\alpha}\alpha]
	-\frac{1}{2}\i_{\alpha}\j_\vphi d\alpha.
\end{align*}
When $\alpha\in \Om_{14}^2$, the normalisation condition
$\mathcal N(\alpha)$ becomes
\begin{equation*}
\begin{split}
=\|X^\bot\|^2_L
+\frac{19}{3} \|X_1\|_L^2
+ \|X_7\|_L^2
-\|X_{27}\|^2_L
+\mathcal R(\a)
+\int_M
g(\j_{\a}\tau,\, \j_\vphi X_{27})\vol_\vphi.
\end{split}
\end{equation*}
\end{lem}
\begin{proof}
Here, we apply \lemref{14 2 properties} to Definition \ref{integration N}., and also use \lemref{last term}, $g(\j_{\a}\tau,\, \j_\vphi d\a)
=\frac{12}{7}g^2(\a, \, \tau)
+g(\j_{\a}\tau,\,\j_\vphi X_{27})
$ and $\frac{12}{7}g^2(\a, \, \tau)=12|X_1|^2$.
\end{proof}

\begin{cor}
If we further require $\alpha\in T\mathcal M_0$, then $X^\bot=X$. 
Thus $$\|X^\bot\|^2_L
+\frac{19}{3} \|X_1\|_L^2+ \|X_7\|_L^2-\|X_{27}\|^2_L=\frac{22}{3} \|X_1\|_L^2+ 2\|X_7\|_L^2.$$
\end{cor}

Besides, if we set $\vphi$ to be torsion free, then both $X_1$ and $X_7$ vanish, by \lemref{14 2 properties}. Thus $\mathcal N=\|X\|^2_L-\|X_{27}\|_L^2$ vanishes identically. 







\subsection{Geodesic concavity}
We compute the Hessian of the volume along the geodesics. 
\begin{thm}\label{Hessian of Vol general} Suppose that $\alpha(t)$ is an $\Om_{14}^2$ geodesic \lemref{142 geodesic}.
Then the formula of the Hessian reads
\begin{equation*}
\begin{split}
3\Vol _{tt}=- \|X^\bot\|^2_L+\mathcal N(\alpha).
\end{split}
\end{equation*}
Additionally, along the $\mathcal N_{14}$ geodesic, the volume is concave.
\end{thm}
\begin{proof}
Differentiating \eqref{eq:1st_vol general} and applying \eqref{eq:2nd_vol general part 1}, we have
\begin{equation*}
\begin{split}
	3\Vol _{tt}=\int_M g(L (\alpha), \alpha_{14}^2) \vol_\vphi+\int_M  \psi\wedge d(\p_t\alpha^2_{14}).
\end{split}
\end{equation*}
We compute the second part
	\begin{align*}
		&\int_M  \psi\wedge d(\p_t\alpha^2_{14})
		=\int_M  \ast\tau\wedge \p_t\alpha^2_{14}
		=\int_M g(\p_t\alpha^2_{14},\tau) \vol_\vphi.
	\end{align*}
Applied the $\Om_{14}^2$ geodesic with $\alpha=\alpha_{14}^2$
\lemref{142 geodesic}, it becomes
	\begin{align*}
		&=-\int_M g\big(\tau,\, g(\tau,\a)\a-\frac{1}{2}\delta(|\a|^2\vphi-\i_\vphi\j_\a\a)-\frac{1}{2}\i_{\a}\j_\vphi d\a \big)\vol_\vphi.
	\end{align*}
We put $g(\vphi,d\tau)=|\tau|^2$ from \cite[Proposition 2.19]{MR4732956} to the 2nd term and 
use the properties of the $\i,\j$ operators in \lemref{i j switch}. We conclude that 
	\begin{align*}
		&=-7\|X_1\|_L^2+\int_M [\frac{1}{2}|\a|^2 |\tau|^2-\frac{1}{2}g(\j_\vphi d\tau,\j_\a\a)+g(\j_{\tau} \a,\j_\vphi d\a )]\vol_\vphi.
	\end{align*}
After inserting the formula of $\mathcal N$ from \lemref{142 geodesic}, we finish the proof.
\end{proof}

\begin{cor}
If two torsion free structures could be connected by a weak $\mathcal N_{14}$ geodesic $\varphi(t)$, then $\alpha(t)$ is generated by diffeomorphisms and the uniqueness question holds.
\end{cor}
\begin{proof}
According to \thmref{Hessian of Vol general}, the volume is concave along $\pi_{14}^2\alpha(t)$. Then the geodesic $\vphi(t)$ consists of torsion free structures. Thus, $\pi_{14}^2\alpha(t)=0$, which means $\p_t\vphi$ is deformed along the $d\Om_7^2$ direction.  
\end{proof}



\subsection{Length contraction}
For each parameter $s$ in $[0,1]$, we are given a family of Laplacian flows $\vphi(t,s), t\geq 0$. 
Along the direction $s$, as shown in Section \ref{14 2 projection}, we choose a family of $2$-forms $\beta(s)$ such that $$\vphi_s=d\beta(s).$$
Fix $t$ and define length $L(t)$ between $\vphi(t,0)$ and $\vphi(t,1)$ by the $L^2$ metric
\begin{align*}
L(t)&:=\int_0^1 \|\vphi_s\|_{\mathcal G}\cdot  ds,\quad \|\vphi_s\|^2_{\mathcal G}=\mathcal G_\vphi(\beta,\beta).
\end{align*}

\begin{thm}\label{length contraction}
The length evolves as following
\begin{align*}
\frac{1}{2}\p_t \|\vphi_s\|^2_{\mathcal G}=-\|X^\bot\|^2_L+\mathcal N(\beta).
\end{align*}
Moreover, the Laplacian flow decreases the length of $\mathcal N$ paths.
\end{thm}
\begin{proof}
We compute the $t$ derivative of $\|\vphi_s\|^2_{\mathcal G}$ by applying the metric-compatibility of the metric $ \mathcal G_\vphi$,
\begin{align*}
\frac{1}{2}\p_t \mathcal G_\vphi(d\beta,d\beta)
=\mathcal G_\vphi(D_t\vphi_s,\vphi_s).
\end{align*}
It contains two parts
\begin{align}\label{distance I II}
I:= \mathcal G_\vphi(\vphi_{ts},\vphi_s),\quad II:= \mathcal G_\vphi(dP(\vphi,\vphi_t,\vphi_s),\vphi_s).
\end{align}
Along the flow, $\vphi_{ts}=(\tri\vphi)_s=\tri\vphi_s+\tri_s\vphi$.
The formula of $\Delta_s \vphi$ is, see \cite[Proposition 3.9]{MR4732956}, 
\begin{align*}
\Delta_s \vphi=
d\delta[\frac{4}{3}g(\vphi_s,\vphi)\vphi
-\frac{1}{2}\i_\vphi\j_\vphi \vphi_s]-d[g(\vphi_s,\vphi) \tau 
-\frac{1}{2}\i_{\tau}\j_\vphi\vphi_s].
\end{align*}

The expression of $I$ becomes
\begin{align*}
&I=\int_M g(\beta,\delta \vphi_s
+\delta[\frac{4}{3}g(\vphi_s,\vphi)\vphi
-\frac{1}{2}\i_\vphi\j_\vphi \vphi_s]
-[g(\vphi_s,\vphi) \tau 
-\frac{1}{2}\i_{\tau}\j_\vphi\vphi_s]) \vol_\vphi\\
&=\int_M[ |d\beta|^2 +\frac{4}{3}g(\vphi_s,\vphi)^2
 -\frac{1}{2}g(\vphi_s,\i_\vphi\j_\vphi \vphi_s)
 -g(\vphi_s,\vphi) g(\beta,\tau) 
 + \frac{1}{2}g(\beta, \i_{\tau}\j_\vphi\vphi_s) ]\vol_\vphi.
\end{align*}
The \lemref{i j switch} gives us that 
$g(\vphi_s,\i_\vphi\j_\vphi \vphi_s)=18|X_1|^2+4|X_{27}|^2$ and
\begin{align*}
I&=\int_M[ \frac{4}{3}|X_1|^2+|X_7|^2
 -|X_{27}|^2
 -g(\vphi_s,\vphi) g(\beta,\tau) 
 +g(\j_\tau\beta, \j_\vphi\vphi_s]) ]\vol_\vphi.
\end{align*}

Inserting the connection \eqref{connection P} into $II$ in \eqref{distance I II}, we have 
\begin{align*}
II
&=\int_M g(\beta,\frac{1}{2}[g(\vphi,d\alpha)\beta+g(\vphi,d\beta)\alpha]\\
		&-\frac{1}{2}\delta[g(\alpha,\beta)\vphi-\i_\vphi \j_{\alpha}\beta]
		-\frac{1}{4}[\i_{\beta}\j_\vphi d\tau+\i_{\alpha}\j_\vphi d\beta]) \vol_\vphi.
\end{align*}
Further using the flow equation $\vphi_t=d\tau$, $g(\vphi,d\tau)=|\tau|^2$ from \lemref{14 2 properties} and \lemref{i j switch}, we obtain that
\begin{align*}
II
=\int_M [\frac{1}{2}|\beta|^2|\tau|^2-\frac{1}{4}g(\beta,\i_{\beta}\j_\vphi d\tau) ]\vol_\vphi.
\end{align*}

In summary, we show that
\begin{align*}
\frac{1}{2}\p_t \|\vphi_s\|^2_{\mathcal G}
= \mathcal L(\alpha) +&\int_M[-g(d\beta ,\vphi) g(\beta,\tau) 
 +g(\j_\tau\beta, \j_\vphi d\beta) \\
&
+\frac{1}{2}|\beta|^2|\tau|^2
		-\frac{1}{2}g(\j_{\beta}\beta,\j_\vphi d\tau) ]\vol_\vphi.
\end{align*}
Thus, we insert the expression of $\mathcal N(\beta)$ in Definition \ref{integration N} to obtain the final expression of $\p_t \|\vphi_s\|^2_{\mathcal G}$.
\end{proof}


According to \thmref{Hessian of Vol general},
along a gauge fixing geodesic, $3\Vol _{tt}$ is
\begin{equation*}
\begin{split}
-\frac{17}{3} \|X_1\|^2_L- \|X_{27}\|^2_L+\int_M[
 \frac{1}{2}|\a|^2|\tau|^2 
-\frac{1}{2}g(d\tau,\,\i_\vphi\j_\a\a)
+g(\j_{\a}\tau,\, \j_\vphi d\a)]\vol_\vphi.
\end{split}
\end{equation*}
It requires further normalisation to get rid of the terms having the torsion involved. Similar calculation happens for the Hessian of the volume along the canonical geodesic. We weill see examples in the next sections. This is the reason, why we work in the larger space $\Om^2$  and apply the normalisation condition $\mathcal N$, instead.

\section{Hyper-symplectic structures}\label{Hypersymplectic structures}
We reminisce about the hyper-symplectic structures.
	Donaldson \cite[Section 2]{MR3932259} suggest  a reduction from a $7$-dimensional $G_2$ manifold to a $4$-dimensional hyper-symplectic manifold $X$.

	We fix $\underline\omega^K$ be a fixed hyper-symplectic triple. The space of the hyper-symplectic triple $\underline\omega=(\om_1,\om_2,\om_3)$ is defined to be the set
	$$\mathcal H_{[\underline\omega^K]}:=\{\underline\omega\text{ is a hyper-symplectic triple cohomologous to }\underline\omega^K\}.$$
	
		Let $\theta^1$, $\theta^2$, $\theta^3$ be the local coordinates on $\mathbb{T}^3$. The $G_2$ structure $\vphi$, defined in \eqref{eq:G2}, lies in the same cohomology class of $\vphi^K:= d\theta^{123}-d\theta^i\wedge \omega^K_i.$
		We define the subset of the space $\mathcal{M}$ of $G_2$-structures in the class $[\vphi^K]$,
		$$\mathcal{M}_{s}:= \{\vphi:= d\theta^{123}-d\theta^i\wedge \omega_i,\quad \underline\omega\in \mathcal H_{[\underline\omega^K]}\}.$$

	The relations between any two of the hyper-symplectic forms are characterised  by the $3\times3$-matrix
		 \begin{equation*}
	\begin{split}
q_{ij}:=\frac{\omega_i\wedge \omega_j}{2\vol_0},\quad u:= \det (q_{kl}),
	\end{split}
	\end{equation*}
where $\vol_0$ is any given volume form on $X$.
Then, the relations are normalised so that $\det (\lambda_{ij})=1$, as the following procedure
	 \begin{equation}\label{Qij}
	\begin{split}
	 \lambda_{ij}:=\frac{\omega_i\wedge \omega_j}{2\vol_4},\quad \lambda_{ij}=u^{\frac{-1}{3}}q_{ij},\quad
	\vol_4:= u^{\frac{1}{3}}  \vol_0 .
	\end{split}
	\end{equation}
	
The integrals $Q_{ij}=\frac{1}{2}\int_X \omega_i\wedge \omega_j=\int q_{ij} \vol_0$ are topological invariants depending on the classes $[\om_i]$ and $[\om_j]$. Together with the arithmetic-geometric mean inequality, they give the upper bound of the volume
\begin{align}\label{volume upper bound}
3\int_Xu^{\frac{1}{3}}  \vol_0\leq \int_X\Tr (q_{ij})  \vol_0=\Tr (Q_{ij}).
\end{align}

	The $G_2$ metric $g_\vphi$ on $\mathbb{T}^3\otimes X^4$ is a product metric $g_\vphi=g_3\oplus g_4$, where
	\begin{equation}\label{g3g4}
	\begin{split}
	g_4(u,v):=\frac{1}{6}\sum_{i,j,k}\eps_{ijk}\frac{i_u \omega_i\wedge i_v\omega_j\wedge \omega_k}{\vol_4}, \quad
	&g_3(\p \theta_i,\p \theta_j):= \lambda_{ij}.
	\end{split}
	\end{equation}
The Levi-Civita symbol $\eps_{ijk}=1$ is the sign of the 3 cyclic permutations of $(123)$.
	Regarding the metric $g_4$, $\omega_i$ is self-adjoint, i.e.
		$*_4\omega_i=\omega_i,$ see \cite{MR3881202}*{Lemma 2.4}.
			
In $\mathbb{T}^3$, the volume form $\vol_3=d\theta^{123}$. Then, applying the definition of the Hodge star
	$
	\a\wedge\ast \b=<\a,\b>_g vol_g,
	$ we have
	\begin{align}\label{hodge 3}
	\ast_3d\theta^i=\frac{1}{2}\eps_{jkl}\lambda^{ij}d\theta^{kl},\quad \ast_3d\theta^{ij}=\eps^{ijl}\lambda_{kl}d\theta^k.
	\end{align}
	We will write $\vol_E:= d\theta^{123}\cdot \vol_0$.


			A $G_2$ structure $\vphi$ is torsion-free if and only if $\vphi$ is both closed and co-closed, i.e.
	 $ d\vphi=0,\ d*_\vphi \vphi=0.$	
		The hyper-symplectic $\underline\omega$ is called torsion-free if the associated $G_2$ structure $\vphi$ \eqref{eq:G2} is torsion-free.
				We call $\underline\omega$ hyper-K\"ahler, if $\lambda=Id$ c.f. \cite{MR2313334}.

	By \eqref{eq:G2} and \eqref{hodge 3}, the dual $4$ form
			\begin{align}\label{psi general}
			\psi:=*_\vphi \vphi=\vol_{4}-\frac{1}{2}\eps_{jkl}\lambda^{ij}d\theta^{kl}\wedge \omega_i.
		\end{align}
		Since $d\vol_4=0$, then
		$
			d\psi=-\frac{1}{2}\eps_{jkl}d\theta^{kl}\wedge (d\lambda^{ij}\omega_i).
		$
		Hence, the co-closedness of $\psi$	
is equivalent to
\begin{align}\label{prop:hyper-symplectic}
\sum_i d\lambda^{ij}\wedge\omega_i=0,\quad \forall j.
\end{align} 	

 \subsection{Hodge Laplacian and torsion form}
 We derive the formula for the relation between the Hodge Laplacian on $X$ and the Hodge Laplacian on $M$, which will be used to write down the expression of the Dirichlet metric.
\begin{lem}[Hodge Laplacian]\label{Laplacian general}
Let $X=-d\theta^i\wedge \Om$ and $\Om$ be any $r$-form on the 4-manifold $X^4$. It holds
	\begin{align*}
	 \delta X=-d\theta^p\wedge \ast_4 (\lambda_{pj}d\lambda^{ij}\wedge \ast_4 \Om)
-(-1)^rd\theta^i\wedge \delta\Om.
	\end{align*}
When $X=-\sum_i d\theta^i\wedge \Om_i$, we have $\Lambda_i(\underline\Om) := \ast_4 ( \lambda_{ij} d\lambda^{pj}\wedge \ast_4 \Om_p)$ and
$$\delta X=-d\theta^i\wedge [\Lambda_i
				+ (-1)^r\delta \Om_i],\quad \Delta X=d\theta^i\wedge d [\Lambda_i
				+(-1)^r \delta \Om_i].$$
\end{lem}
\begin{proof}
We compute with the help of \eqref{hodge 3}
			\begin{align*}
				*X=-\ast_3 d\theta^i\wedge \ast_4 \Om
				=-\frac{1}{2}\eps_{jkl}\lambda^{ij}d\theta^{kl}\wedge \ast_4 \Om .
			\end{align*}
Taking the exterior derivative, we obtain
$$
d* X
=-\frac{1}{2}\eps_{jkl}  d\theta^{kl} \wedge d\lambda^{ij}\wedge \ast_4 \Om
-\frac{1}{2}\eps_{jkl}\lambda^{ij}  \cdot d\theta^{kl}\wedge d(\ast_4 \Om).
$$	
By \eqref{hodge 3} again, we further have compute $\delta X=(-1)^{r+1}*(d* X)$
\begin{equation*}
\begin{split}
=\frac{(-1)^{r}}{2}\eps_{jkl} \ast[ d\theta^{kl} \wedge d\lambda^{ij}\wedge \ast_4 \Om]
+\frac{(-1)^{r}}{2}\eps_{jkl}\lambda^{ij}  \ast[ d\theta^{kl}\wedge d(\ast_4 \Om)].
\end{split}
\end{equation*}
While, both $d\lambda^{ij}\wedge \ast_4 \Om$ and $d(\ast_4 \Om)$ are $(5-r)$-forms, we take Hodge star and multiply the identity above with $(-1)^{5-r}$, that leads to
\begin{equation*}
\begin{split}
-\frac{1}{2}\eps_{jkl} (\eps^{klq}\lambda_{pq}d\theta^p)\wedge \ast_4 (d\lambda^{ij}\wedge \ast_4 \Om)
-\frac{1}{2}\eps_{jkl}\lambda^{ij}(\eps^{klq}\lambda_{pq}d\theta^p)\wedge \ast_4 d(\ast_4 \Om).
\end{split}
\end{equation*}
Then, we insert $\eps_{jkl}\eps^{klq}=\delta^j_{q}$ and $\lambda^{ij}\lambda_{pj}=\lambda^{ij}\lambda_{jp}=\delta^i_{p}$ by symmetry of the matrix $\lambda$, to get
\begin{equation*}
\begin{split}
\delta X&=-d\theta^p\wedge \ast_4 (\lambda_{pj}d\lambda^{ij}\wedge \ast_4 \Om)
d\theta^i\wedge \ast_4 d(\ast_4 \Om).
\end{split}
\end{equation*}
Finally, we obtain the Hodge Laplacian from $\Delta X=d\delta X.$
\end{proof}
Because the $2$-form $\om_i$ is self-dual and closed, we have $\delta\om_i=0$. \lemref{Laplacian general} immediately implies that
\begin{cor}[Torsion]\label{torsion form and d torsion}
The torsion form of the $G_2$ structure \eqref{eq:G2} has the expression
\begin{align*}
\tau=\sum_id\theta^i\wedge\tau_i,\quad \tau_i=\ast_4 ( \lambda_{ij}  d\lambda^{pj}\wedge\omega_p).
\end{align*}
\end{cor}
\subsection{Tangent vectors and potentials}\label{canonical forms}
We restrict the construction in \cite{MR4732956} to the subspace $\mathcal{M}_s$ of the space of all $G_2$ structures.

	Let $\vphi(t)$ be a smooth path in $\mathcal{M}_s$ and $X:=  \p_t\vphi(t)$.
		Differentiating \eqref{eq:G2}, we get the tangent vector
	\begin{align}\label{vphit general}
		X=-\sum_id\theta^i\wedge d \alpha_i,\quad
		( \omega_i)_t=d \alpha_i.
	\end{align}

We propose in \cite{MR4732956} a canonical way to find a representative,	 solving the inverse of the Laplacian of $X$, which is called the potential of $X$, i.e., a $d$-exact $3$-form $\mathbf  u$ solving $$\tri \mathbf u=X.$$ According to the Hodge theory, $\mathbf u$ is uniquely determined, as $X$ is $d$-exact.

We write $\mathbf u=d\boldsymbol\mu$ for some $2$-form $\boldsymbol\mu$ on $M$ and $\boldsymbol{\alpha_c}=\delta \mathbf u$. The notations in bold refer to the lift of quantities from $X^4$ to the 7-manifold $M$.

The $2$-form $\boldsymbol\mu$ has three possibilities, a $2$-form on $\mathbb T^3$, a $2$-form on $X^4$, or the mixed one $\boldsymbol\mu=\sum_id\theta^i\wedge \mu_i$ with a $1$-form $ \mu_i$ on $X^4$. 	
Utilising \lemref{Laplacian general} and comparing with the form of $X$ \eqref{vphit general}, we see that it suffices to solve the mixed one.

In summary, we look for a particular solution, which is a $d$-exact $3$-form
\begin{align}\label{bf u and ui}
\mathbf u=d\boldsymbol\mu=-d\theta^i\wedge  u_i,
\quad  u_i:= d \mu_i,
\end{align}
satisfying the Hodge-Laplacian Laplacian equation
$$ X=\Delta \mathbf  u=d\theta^i\wedge d (\Lambda_i+ \delta   u_i).$$

\begin{defn}[Canonical $1$-form]\label{vphit general alphai}
We define the canonical $1$-form
\begin{align*}
(\alpha_c)_i:=- (\Lambda_i + \delta   u_i),\quad \Lambda_i = \ast_4 ( \lambda_{ij} d\lambda^{pj}\wedge \ast_4 u_p).
\end{align*}	
\end{defn}
Therefore, the tangent vector $X$ could be rewritten as a $d$-exact form
\begin{align}\label{vphit general alpha bold}
X=d\boldsymbol{\alpha_c},\quad
\boldsymbol{\alpha_c}=\delta \mathbf u=d\theta^i\wedge(\alpha_c)_i.
\end{align}
\begin{rem}
In general, $\alpha_i$ in \eqref{bf u and ui} is not uniquely determined, we refer to Definition \ref{L2 connection}. 
The choice of $\alpha_i$ depends on the different normalisations, since adding any $d$-closed $1$-form to $\alpha_i$ does not change the vector field $X$.
\end{rem}
\subsection{Inherited metrics}\label{inherited metric}
\begin{lem}[Inherited metrics]		\label{eq:Dmetric 3}	
Let $Y=-\sum_id\theta^i\wedge d\beta_i$ be another tangent vector. When restricted in $\mathcal{M}_s$, the $L^2$ and Dirichlet metrics have the following expressions, respectively:
\begin{enumerate}
\item
$\int_M g_\vphi( \boldsymbol{\alpha}, \boldsymbol{\beta})  \vol_\vphi
=\int_M\sum_{i,j}  \lambda^{ij} g_4( \alpha_i, \beta_j) u^{\frac{1}{3}} \vol_E$;
\item
$\int_M \sum_{i,j} \lambda^{ij}g_4( u_i,d\beta_j) u^{\frac{1}{3}} \vol_E
=\int_M\sum_{i,j}  \lambda^{ij} g_4( (\alpha_c)_i, (\beta_c)_j) u^{\frac{1}{3}} \vol_E.
$
\end{enumerate}
\end{lem}
\begin{proof}
The first identity follows from plugging the metric $g_3$ in \eqref{g3g4} into the expression of the $L^2$ metric \eqref{L2 metric}.
The Dirichlet metric \eqref{metric} has an equivalent expression
		\begin{equation}\label{gradient metric G}
		\begin{split}
		\mathcal G_\vphi(X,Y)
	=\int_Mg_\vphi(GX,Y)  \vol_\vphi.
				\end{split}
			\end{equation}
Inserting \eqref{bf u and ui} and \eqref{vphit general alpha bold}, we have the second identity.
\end{proof}

\subsection{Weighted volume functional}
Now, we introduce the weighted volume functional, which generalises Hitchin's volume functional.
\begin{defn}[$\chi$ volume functional]\label{weighted volume functional 1}
	We let $\chi$ be a nonnegative function.
Let $ \chi:= (\cdot)^{\frac{1}{3}}\tilde \chi(\cdot)$ be a non-negative concave function. Then, we introduce a $\chi$ volume functional
$
\Vol_\chi(\vphi):=\int_M \tilde \chi(u) \vol_\vphi=\int_M  \chi(u) \vol_E.
$
When $\tilde \chi(\cdot)=1$, it is Hitchin's volume functional.
\end{defn}

\begin{lem}The variations of the weighted functional are
\begin{equation}\label{eq:2nd_vol chi general}
\begin{split}
	(\Vol_\chi)_{t}=\int_M\chi_{u} u_{t} \vol_E; 	(\Vol_\chi)_{ts}=\int_M(\chi_{uu}u_tu_s  + \chi_{u} u_{ts}) \vol_E.
\end{split}
\end{equation}
\end{lem}

From the point of view of $G_2$ geometry, we have another expression for the variations of the volume functional, see Section \ref{Variations and linearised operators}.

In this article, we provide examples, some of which answer this question (Section \ref{HS and OT} and \ref{concavity}), and some have a negative answer to this question (Section \ref{non-concavity}).

Concerning the geometric flows, since $\p_t\vol_\vphi = \frac{1}{3}u^{\frac{-2}{3}} u_t \vol_E$, and from $G_2$ geometry, the variation of the volume form is $
	\p_t\vol_\vphi=\frac{1}{3}g( \vphi_t ,\vphi) \vol_\vphi$, we have $u_t=u g( \vphi_t ,\vphi) $ and
\begin{equation}\label{vol chi t}
\begin{split}
	(\Vol_\chi)_{t}=\int_M g( \vphi_t , \chi_{u} u^{\frac{2}{3}} \vphi)\vol_\vphi.
\end{split}
\end{equation}
Under the Dirichlet metric \eqref{metric}, the Euler-Lagrange equation of the $\chi$  volume and the gradient flow are
\begin{equation}\label{chi flow general}
\begin{split}
d\delta(\chi_{u} u^{\frac{2}{3}} \vphi)=0,\quad	\vphi_t=d\delta (\chi_{u} u^{\frac{2}{3}} \vphi).
\end{split}
\end{equation}


\begin{rem}
When $\chi(u)=3u^{\frac{1}{3}}$, it becomes the $G_2$ Laplacian flow,
\begin{align}\label{Laplacian flow}
\vphi_t=d\delta \vphi=d\tau,
\end{align} which is the gradient flow of the volume functional under the Dirichlet metric
$
\Vol(\vphi)=\int_M u^{\frac{1}{3}}\cdot \vol_E .
$
Written in the hyper-symplectic form $\underline \om$, the flow equation is
\begin{align}\label{Laplacian flow form}
(\om_i)_t=-d\tau_i,
\end{align} which is the hyper-symplectic flow \cite{MR3881202}.
\end{rem}

\section{Mixed structures}\label{Mixed structures}
We consider the following three hyper-symplectic structures with two of its components being K\"ahler
	\begin{align}\label{HS vector}
		\underline\om_1:=(\om_1,\om^K_2,\om^K_3),\quad\underline\om_2:=(\om^K_1,\om_2,\om^K_3),\quad\underline\om_3:=(\om^K_1,\om^K_2,\om_3).
	\end{align}
	
Following the construction in Section \ref{HS and G2}, for each $\underline\om_i$ we could define the matrix $(q^i_{kl})$ \eqref{Qij} and the $G_2$ structure $\vphi_i$ \eqref{eq:G2}.
We select one of these triples
\begin{align}\label{mix HS}
\underline{\omega}_3=(\omega^K_1,\omega^K_2,\omega_3).
\end{align}
Direct computation shows that the corresponding symmetric matrix \eqref{Qij} is
\begin{align}\label{mix Q}
\mathcal Q:=(q_{kl}) =
\begin{bmatrix}
1& 0 & q_{1} \\
0& 1& q_{2} \\
q_{1} & q_{2} & q_{3}
\end{bmatrix},\quad \lambda_{kl}=u^{-\frac{1}{3}}q_{kl},
\end{align}
where we use the notations
\begin{align}\label{Qijsimple}
q_1:= q_{13}=q_{31},\quad q_2:= q_{23}=q_{32},\quad q_3:= q_{33}.
\end{align}
The determinant of $\mathcal Q$ is
\begin{align}\label{uq3}
u:= \det (q_{kl})= q_{3} - q_{1}^2 - q_{2}^2>0,
\end{align} which is bounded above by $q_{3}=\frac{\omega_3\wedge \omega_3}{2\vol_0}.$
Its inverse is
\begin{align}\label{mix Q inverse}
\mathcal Q^{-1}=\frac{1}{u}
\begin{bmatrix}
q_{3}-q_{2}^{2}&q_{1}q_{2}&-q_{1}\\
q_{1}q_{2}&q_{3}-q_{1}^{2}&-q_{2}\\
q_{1}&-q_{2}&1
\end{bmatrix},\quad \lambda^{ij}=u^{\frac{1}{3}}q^{ij}.
\end{align}
In particular, $\lambda^{33}=u^{\frac{-2}{3}}$.
\subsubsection{Tangent vectors}
The $3$-form, learnt from \eqref{eq:G2}, is
$$\vphi:= d\theta^{123}-\sum_{i=1,2}d\theta^i\wedge \omega^K_i-d\theta^3\wedge \omega_3.
$$ 
The tangent vector \eqref{vphit general} becomes
\begin{align}\label{vphit general 1}
\p_t\omega_3=d\alpha,\quad X=d(d\theta^3\wedge \alpha ).
\end{align}
\subsubsection{Canonical $1$-form}
The potential $\mathbf u$ of $X$ in \eqref{vphit general alpha bold} has simpler expressions
\begin{align}\label{vphit general 1 potential form}
\mathbf u=-d\theta^3\wedge  u_3.
\end{align}

In general, the $1$-form $\alpha$ in \eqref{vphit general 1}, representing the variation of the symplectic form $\om_3$, could be chosen to be any $1$-form differ from the canonical one by a $d$-closed form, which is defined in Definition \ref{vphit general alphai}
\begin{align*}
(\alpha_c)_3=-(\Lambda_3 + \delta   u_3),\quad \Lambda_3 = \ast_4 ( \lambda_{3j} d\lambda^{3j}\wedge \ast_4 u_3),
\end{align*}
whose lift on $X$ is $\boldsymbol{\alpha_c}=\delta \mathbf u$ \eqref{vphit general alpha bold}. In which, we compute that
\begin{align*}
\sum_j\lambda_{3j} d\lambda^{3j}=-\frac{2}{3}d\log u-\frac{d(q_1^2+q_2^2)}{2u}.
\end{align*}
\subsubsection{$L^2$ metric}
We use the notations $v_3,w_3$ and $\beta,\gamma$ representing the quantities on $X$
and the bold characters on $M$.

\begin{lem}\label{metric mixed}
Let $Y=-d\theta^3\wedge d\beta$ and $Z=-d\theta^3\wedge d\gamma$.
The inherited $L^2$ metric $\mathcal G_\vphi(Y,Z)$ in \lemref{eq:Dmetric 3} is reduced to	the inner produce between $1$-forms
$$\int_M g_4( \beta,\gamma) u^{\frac{-1}{3}} \vol_E.$$
\end{lem}
\begin{rem}
In the following sections, we consider examples that further simplify the inner product between two $1$-forms: $g_4(\alpha,\beta)$ on the 4-manifold $X^4$. For example, when $g_4$ is diagonal,
\begin{align*}
\Lambda_3 = -\frac{2}{3} \ast_4 (d\log u\wedge \ast_4 u_3).
\end{align*}
\end{rem}

\subsection{Connection and geodesic}
We further expand the explicit formula of the connection in Definition \ref{L2 connection}. After reduction, we sre now working on $1$-forms, instead of 2-forms.
Let $\vphi(t,s,r)$ be a 3-parameter family and $X=\p_t\vphi=-d\theta^3\wedge d\alpha$, $Y=\p_s\vphi=-d\theta^3\wedge d\beta$, $Z=\p_r\vphi=-d\theta^3\wedge d\gamma$. Also, $\boldsymbol{\alpha}=d\theta^3\wedge\a$, $\boldsymbol{\b}=d\theta^3\wedge\b$, $\boldsymbol{\gamma}=d\theta^3\wedge\gamma$.
\begin{defn}\label{connection om}
For any tangent vectors $x,y\in TX^4$, we denote 
\begin{equation*}
\begin{split}
	i_x\Om_{i}:=\frac{1}{2}\sum_{j,k}\eps_{ijk}i_x\omega_j \wedge \omega_k,\quad i_{x,y} \Om_i:=\frac{1}{2}\sum_{j,k}\eps_{ijk} i_x \omega_j\wedge i_y \omega_k
\end{split}
\end{equation*} 
and $A^\alpha_i(x,y):= \frac{1}{ u^{\frac{1}{3}} \vol_0}[i_xd\alpha_i\wedge i_y\Om_i
+ i_y d\alpha_i \wedge i_x\Om_i+ i_{x,y} \Om_i\wedge d\alpha_i]$.
\end{defn}
\begin{defn}\label{connection formula}
We define the connection $D_\a\b:= \b_t+P(\a,\b)$,
\begin{equation*}
\begin{split}
& P(\a,\b):= \frac{-1}{6}  \{i_{\beta^\sharp} A^\alpha_3+
 i_{\alpha^\sharp} A^\beta_3+u^{\frac{2}{3}}\delta[\ast_4(u^{-\frac{2}{3}}  i_{\alpha^\sharp,\beta^\sharp} \Om_3)]\\
&-u^{\frac{2}{3}}\delta[\alpha\wedge \ast_4(u^{-\frac{2}{3}} i_{\beta^\sharp}\Om_3)]
-u^{\frac{2}{3}}\delta[\beta\wedge \ast_4(u^{-\frac{2}{3}} i_{\alpha^\sharp}\Om_3)]\}.
\end{split}
\end{equation*} 
The notations are further specified in \lemref{ixalpha} and \lemref{pg} and the notion $^\sharp$ is the musical isomorphism. 
\end{defn}
This is a special case of Definition \ref{L2 connection}.
\begin{thm}\label{connection free compatible}
The connection $D_\alpha\beta$ is a Levi-Civita connection with respect to the metric in \lemref{metric mixed}. 

Precisely, we have $P=P^A-K$, see $D^A$ in Definition \ref{connection PA} and $K$ in Definition \ref{connection S}.
\end{thm}
\begin{proof}
The identity is a combination of Definition \ref{connection PA} for the compatible, but may not be torsion free, connection $ D^A$ and Definition \ref{connection S} for the contorsion tensor $ K$.
Torsion freeness follows from the symmetry of the expression on $\alpha$ and $\beta$. The metric compatibility of $D$ holds because of the metric compatibility of $D^A$ in \lemref{connection PA compatible} and $$\mathcal G_\vphi(  K(\alpha,\beta),\gamma) +\mathcal G_\vphi (\beta,  K(\alpha,\gamma))=0,$$ which is a consequence of the relation between the torsion $T^A$ of $D^A$ and the contorsion tensors
$\mathcal G_\vphi(T^A(\alpha,\beta), \gamma)
=\mathcal G_\vphi(\alpha, S(\beta,\gamma)).$
\end{proof}

\begin{cor}[$L^2$ Geodesic for the mixed hyper-symplectic structures]\label{geodesic HS cor}
The $L^2$ geodesic equation $  D_\alpha\alpha=0$ is 
\begin{equation*}
\begin{split}
  6  \alpha_t=2i_{\alpha^\sharp} A^\alpha_3
&+u^{\frac{2}{3}}\delta[\ast_4(u^{-\frac{2}{3}} i_{\alpha^\sharp}\om_1\wedge i_{\alpha^\sharp}\om_2 )]
-2u^{\frac{2}{3}}\delta[\alpha\wedge \ast_4(u^{-\frac{2}{3}}  i_{\alpha^\sharp} \omega_1\wedge \omega_2)].
\end{split}
\end{equation*} 
\end{cor}
\begin{proof}
Using \lemref{ixalpha} and $i_{\alpha^\sharp}i_{\alpha^\sharp}=0$, we obtain
\begin{align}\label{i double}
i_{\alpha^\sharp,\alpha^\sharp} \Om_3=i_{\alpha^\sharp}\om_1\wedge i_{\alpha^\sharp}\om_2 
\end{align}
and thus the geodesic equation from Definition \ref{connection formula}.
\end{proof}

We close this subsection by computing the factors in Definition \ref{connection om}.
\begin{lem}\label{ixalpha}
When $i=3$, we further have
\begin{align*}
  i_x\Om_3
  =  i_x \omega_1\wedge \omega_2,\quad    i_{x,y} \Om_3=i_y\om_1\wedge i_x\om_2+\frac{1}{2}(i_yi_x\om_1) \om_2+\frac{1}{2}(i_yi_x\om_2 )\om_1 .
\end{align*}
\end{lem}
\begin{proof}
We recall that given a vector field $W$, a $p$-form $\alpha$ and a $q$-form $\beta$, the contraction and the wedge product have the relation
\begin{align*}
i_W(\alpha\wedge\beta)=i_W\alpha\wedge\beta+(-1)^p\alpha\wedge i_W\beta.
\end{align*}
When $p=q=2$ and $V$ is another vector field,
\begin{align*}
i_{V,W}(\alpha\wedge\beta)=i_{V,W}\alpha\wedge\beta
-i_{W}\alpha\wedge i_{V}\beta
+i_{V}\alpha\wedge i_{W}\beta
+\alpha\wedge i_{V,W}\beta.
\end{align*}

We use
$
0=i_y (\omega_1\wedge \omega_2)
=i_y \omega_1\wedge \omega_2+ \omega_1\wedge i_y\omega_2
$	to see that
	\begin{equation*}
	\begin{split}
&   i_y\Om_3
=\frac{1}{2}\sum_{j,k}\eps_{3jk}  i_y \omega_j\wedge \omega_k
=  i_y \omega_1\wedge \omega_2.
	\end{split}
	\end{equation*}
Similarly, we have
$ i_x\Om_3
  =  i_x \omega_1\wedge \omega_2
$
and $  i_{x,y} \Om_3
=\frac{1}{2}\sum_{j,k}\eps_{3jk} i_x \omega_j\wedge i_y \omega_k
= \frac{1}{2} ( i_x \omega_1\wedge i_y \omega_2- i_x \omega_2\wedge i_y \omega_1).
$
\end{proof}

\subsubsection{Variation of $g_4$}
The main technical part to obtain the connection is the variation of the metric $g_4$ on the $4$-manifold $X^4$.

Taking $t$-derivative of \eqref{Qijsimple}, we obtain the variation of the matrix $\mathcal Q$. 
\begin{lem}\label{omega d alpha}
$\omega_1\wedge d\alpha=2(q_1)_t\vol_0,\quad 
\omega_2\wedge d\alpha=2(q_2)_t\vol_0,\quad
 \omega_3\wedge d\alpha=(q_3)_t\vol_0$.
\end{lem}	

\begin{lem}\label{pg}
$
\p_t g_4(x,y)=\frac{1}{3}\sum_{i} A_i(x,y)
	-\frac{1}{3}(\log u)_t g_4(x,y) ,
$ where
$A^\alpha_1=A^\alpha_2=0$ and the symmetric $2$-tensor $$
A^\alpha_3(x,y) = \frac{1}{ u^{\frac{1}{3}} \vol_0}[
i_x d\alpha\wedge i_y \omega_1\wedge \omega_2
+ i_y d\alpha\wedge i_x \omega_1\wedge \omega_2+d\alpha\wedge  i_{x,y} \Om_3].$$
Furthermore, 
$
(\p_tg^{-1}_4)( \beta,\gamma)=
-\frac{1}{3} A^\alpha_3(\beta^\sharp,\gamma^\sharp)
	+\frac{1}{3}(\log u)_t g_4(\beta^\sharp,\gamma^\sharp).
$ 
\end{lem}
\begin{proof}
The $t$-derivative of $g_4$ is obtained from \eqref{g3g4}, \eqref{Qij} and \eqref{vphit general}. The formula of $\p_t g_4(x,y)$ is
	\begin{equation*}
	\begin{split}
&\frac{1}{6}\sum_{i,j,k}\eps_{ijk}\frac{i_x d\alpha_i\wedge i_y \omega_j\wedge \omega_k}{u^{\frac{1}{3}}  \vol_0}
	+\frac{1}{6}\sum_{i,j,k}\eps_{ijk}\frac{i_x \omega_i\wedge i_y d\alpha_j\wedge \omega_k}{u^{\frac{1}{3}}  \vol_0}\\
	&+\frac{1}{6}\sum_{i,j,k}\eps_{ijk}\frac{i_x \omega_i\wedge i_y \omega_j\wedge d\alpha_k}{u^{\frac{1}{3}}  \vol_0}
	-\frac{1}{3}g_4(x,y) (\log u)_t.
	\end{split}
	\end{equation*}
From the formula of the tangent vector \eqref{vphit general 1}. we have $\alpha_1=\alpha_2=0$ and $\alpha_3=\alpha$. Thus, we prove the lemma, by plugging \lemref{ixalpha} into Definition \ref{connection om}.

Under the local coordinates, we recall the contraction operator $i_\beta(\p_tg_4)=g_4^{il}(\p_tg_4)_{lk}\beta_i$ and compute that
\begin{align*}
(\p_tg^{-1}_4)( \beta,\gamma)=\p_t(g_4^{ij})\beta_i\gamma_j
=-g_4^{il}g_4^{kj}\p_t({g_4}_{kl})\beta_i\gamma_j 
=-\p_t({g_4})(\beta^\sharp,\gamma^\sharp),
\end{align*}
which implies the formula of the derivative of the inverse metric.
\end{proof}

\subsubsection{Connection $ D^A$}
We differentiate the inherited metric in \lemref{metric mixed} and insert \lemref{pg}. So, we get $\p_t\mathcal G_\vphi(\beta,\gamma)-\mathcal G_\vphi(\beta_t,\gamma)-\mathcal G_\vphi(\beta,\gamma_t)$
\begin{equation*}
\begin{split}
&=\int_M \{(\p_tg^{-1}_4)( \beta,\gamma) 
-\frac{1}{3} g_4( \beta^\sharp,\gamma^\sharp) (\log u)_t \} u^{\frac{-1}{3}} \vol_E
=\int_M -\frac{1}{3} A^\alpha_3(\beta^\sharp,\gamma^\sharp)  u^{\frac{-1}{3}} \vol_E.
\end{split}
\end{equation*} 
In order to find a metric compatible connection $P^A$, the identity above should be equal to $\mathcal G_\vphi(  P^A(\alpha,\beta),\gamma) +\mathcal G_\vphi (\beta,  P^A(\alpha,\gamma))$. 
\begin{defn}\label{connection PA}
We define the connection $  D^A_\alpha\beta:= \beta_t+ P^A(\alpha,\beta)$, 
\begin{align*}
P^A(\alpha,\beta):
=-\frac{1}{6} i_{\beta^\sharp} A^\alpha_3.
\end{align*}
\end{defn}
\begin{lem}\label{connection PA compatible}
$ D^A$ is metric compatible and its torsion 
\begin{align*}
T^A(\alpha,\beta):= P^A(\alpha,\beta)-P^A(\beta,\alpha)=-\frac{1}{6} i_{\beta^\sharp} A^\alpha_3+\frac{1}{6} i_{\alpha^\sharp} A^\beta_3.
\end{align*}
is a $1$-form value antisymmetric $2$-tensor.
\end{lem}


\subsubsection{Contorsion tensor $ K$}
\begin{lem}\label{connection xy}
$\mathcal G_\vphi( P^A(\alpha,\beta), \gamma)=\int_{M} g_4( \alpha, S_1(\beta,\gamma)) u^{\frac{-1}{3}} \vol_E $,
where $$S_1(\beta,\gamma):= -\frac{1}{6} u^{\frac{2}{3}} \delta [\beta\wedge \ast(u^{-\frac{2}{3}} i_\gamma\Om_3\wedge d\theta^{123})
+\gamma\wedge \ast(u^{-\frac{2}{3}} i_\beta\Om_3\wedge d\theta^{123})
-\ast(u^{-\frac{2}{3}}  i_{\beta,\gamma} \Om_3\wedge d\theta^{123})]$$
is a $1$-form value symmetric $2$-tensor.
\end{lem}
\begin{proof}
We continue the computation from Definition \ref{connection PA}
\begin{equation*}
\begin{split}
&\int_M g_4( P^A(\alpha,\beta),\gamma) u^{\frac{-1}{3}} \vol_E
=-\frac{1}{6}\int_M g_4( i_{\beta^\sharp} A^\alpha_3,\gamma) u^{\frac{-1}{3}} \vol_E\\
&=-\frac{1}{6}\int_{M} u^{-\frac{2}{3}}  [i_{\beta^\sharp} d\alpha \wedge i_{\gamma^\sharp}\Om_3
+ i_{\gamma^\sharp} d\alpha \wedge i_{\beta^\sharp}\Om_3+d\alpha\wedge i_{\beta^\sharp,\gamma^\sharp} \Om_3] d\theta^{123}.
\end{split}
\end{equation*} 
We compute with the help of the integration by parts repeatedly.

The first term
$\int_{M} i_{\beta^\sharp} d\alpha \wedge [ u^{-\frac{2}{3}}  i_{\gamma^\sharp}\Om_3
\wedge d\theta^{123}]$
\begin{equation*}
\begin{split}
&= \int_{M} g( i_{\beta^\sharp}  d\alpha, \ast (u^{-\frac{2}{3}} i_{\gamma^\sharp}\Om_3\wedge d\theta^{123}) )\vol\\
&=\int_{M} g( d\alpha, \beta\wedge \ast (u^{-\frac{2}{3}} i_{\gamma^\sharp}\Om_3\wedge d\theta^{123}))\vol\\
&=\int_{M} g( \alpha, \delta[\beta\wedge \ast (u^{-\frac{2}{3}} i_{\gamma^\sharp}\Om_3\wedge d\theta^{123})])\vol.
\end{split}
\end{equation*} 

Similarly, the second term is direct 
$$
\int_{M} u^{-\frac{2}{3}} i_{\gamma^\sharp} d\alpha \wedge i_{\beta^\sharp}\Om_3\wedge d\theta^{123}
=\int_{M} g_4( \alpha, \delta[\gamma\wedge \ast(u^{-\frac{2}{3}} i_{\beta^\sharp}\Om_3\wedge \theta^{123})])\vol.
$$

The third term is $\int_{M} u^{-\frac{2}{3}} d\alpha\wedge  i_{\beta^\sharp,\gamma^\sharp} \Om_3\wedge d\theta^{123}$
\begin{equation*}
\begin{split}
&=\int_{M} g(d\alpha,\ast(u^{-\frac{2}{3}}  i_{\beta^\sharp,\gamma^\sharp} \Om_3\wedge d\theta^{123}))\\
&=\int_{X^4} g(\alpha,\delta[\ast(u^{-\frac{2}{3}}  i_{\beta^\sharp,\gamma^\sharp} \Om_3\wedge d\theta^{123})]).
\end{split}
\end{equation*} 
Thus, we prove the lemma after adding these identities together.
\end{proof}

\begin{lem}\label{connection yx}
$\mathcal G_\vphi(  P^A(\beta,\alpha), \gamma)=\int_{M} g_4( \alpha, S_2(\beta,\gamma)) u^{\frac{-1}{3}} \vol_E$,
where $$S_2(\beta,\gamma):= -\frac{1}{6} i_{\gamma^\sharp} A^\beta_3   .$$
\end{lem}
\begin{proof}We proceed as the lemma above 
\begin{equation*}
\begin{split}
\int_M g_4( P^A(\beta,\alpha),\gamma) u^{\frac{-1}{3}} \vol_E
=-\frac{1}{6}\int_M g_4( i_{\alpha^\sharp} A^\beta_3,\gamma) u^{\frac{-1}{3}} \vol_E.
\end{split}
\end{equation*} 
By switching the positions of $\alpha$ and $\gamma$, we prove this lemma.
\end{proof}

\begin{defn}\label{connection S}
We define the symmetric tensor 
$
S(\beta,\gamma):= S_1(\beta,\gamma)-S_2(\beta,\gamma)
$ 
and the contorsion tensor 
\begin{equation*}
\begin{split}
K(\alpha,\beta):= \frac{1}{2}[T^A(\alpha,\beta)+S(\alpha,\beta)+S(\beta,\alpha)].
\end{split}
\end{equation*} 
\end{defn}

\subsection{Hessian of volume}	
We explore the second variation of the volume functional \thmref{Hessian of Vol general}, for the $G_2$ structures constructed from the hyper-symplectic structures, by using the geodesic equation \lemref{geodesic HS cor} under local coordinates.
\begin{defn}
We define
\begin{align*}
&A_{i_1i_2,j_1j_2}:=
\delta_{j_1   i_{1} i_{2}}^{k_1 k_2k_3  } (\om_1)_{k_1j_2} {(\om_2)_{k_2k_3}}
  -2  (\om_1)_{i_1j_1} (\om_2)_{i_2j_2},\\
  &B_{i_1i_2i_3i_4,j_1j_2}:=
  [{(\om_1)}_{j_1 i_2}  \alpha_{j_2 i_1} 
+      (\om_1)_{j_2 i_2 }     \alpha_{ j_1 i_1} ]  (\omega_2)_{i_3i_4} 
+       (\om_1)_{j_1 i_1}   (\om_2)_{j_2 i_2}    \alpha_{i_3i_4}.
\end{align*}
\end{defn}

\begin{thm}\label{Vtt diagonal}
If $\mathcal Q$ is diagonal, then along the $L^2$ geodesic in \corref{geodesic HS cor},
  \begin{equation*}
\begin{split}
18\Vol _{tt}=\int_M -4(\log u)_t^2 \vol_\vphi+I_1-I_2.
\end{split}
\end{equation*}
Here, we write the integrands
  \begin{equation*}
\begin{split}
&I_1:= \frac{2}{3} (\log u)^{j_3} 
 {(\omega_{3})_{ j_3}}^{j_1}\, u^{\frac{-2}{3} }\, \alpha^{j_2} 
 B_{i_1i_2i_3i_4,j_1j_2}\,dx^{i_1i_2i_3i_4}\wedge d\theta^{123},\\
&I_2:=\frac{1}{3} (\log u)^{j_3} [ (\omega_{3})_{ j_3i_4}
(u^{\frac{-2}{3} }
\alpha^{j_1}\alpha^{j_2}  A_{i_1i_2,j_1j_2})_{i_3}
\,dx^{i_1i_2i_3i_4}\wedge d\theta^{123}.
\end{split}
\end{equation*}
\end{thm}
When $\mathcal Q$ is diagonal, the torsion form in \corref{torsion form and d torsion} is 
$$\tau_3=u^{\frac{2}{3}} \ast_4\om_u.$$

In Section \ref{Torus fibration}, we will further derive the canonical form via solving some nonlinear equation and apply the canonical geodesic to the formula of the Hessian. 
\subsection{Proof of \thmref{Vtt diagonal}}
The rest of this section is devoted to the proof via expanding the second order derivative of the volume along the geodesic
\begin{equation*}
\begin{split}
	3\Vol_t=\int_M g( \vphi_t ,\vphi) \vol_\vphi,\quad
	3\Vol _{tt}=\int_M \psi_t\wedge \vphi_t+ \psi\wedge \vphi_{tt}.
\end{split}
\end{equation*}
\begin{lem}
$\psi_t\wedge \vphi_t=-\{\frac{2}{3}(\log u)_t^2+ 2 u^{-1} [(q_1)^2_t+(q_2)^2_t] \}\vol_\vphi$.
\end{lem}
\begin{proof}
Differentiating \eqref{psi general} gives
\begin{align}\label{psi t general 1}
\psi_t=\frac{1}{3} u^{-\frac{2}{3}} u_t \vol_0-\frac{1}{2}\eps_{jkl}(\lambda^{ij})_td\theta^{kl}\wedge \omega_i
-\frac{1}{2}\eps_{jkl}\lambda^{3j}d\theta^{kl}\wedge d\alpha.
		\end{align}
		
We compute from \eqref{vphit general 1} and \eqref{psi t general 1} that
$\psi_t\wedge \vphi_t$ is
\begin{equation*}
\begin{split}
&=\psi_t
\wedge
[-d\theta^3\wedge d\alpha ]
=
\frac{1}{2}\eps_{jkl}(\lambda^{ij})_td\theta^{kl3}\wedge \omega_i
\wedge d\alpha 
=
(\lambda^{i3})_td\theta^{123}\wedge \omega_i
\wedge d\alpha .
\end{split}
\end{equation*}
From \eqref{mix Q inverse}, $\lambda^{i3}=-q_i u^{-\frac{2}{3}}$, $i=1,2$, $\lambda^{33}=u^{-\frac{2}{3}}$. We conclude from \lemref{omega d alpha} that
\begin{equation*}
\begin{split}
(\lambda^{i3})_t \omega_i
\wedge d\alpha 
&=[-(u^{-\frac{2}{3}}q_1)_t2(q_1)_t-(u^{-\frac{2}{3}}q_2)_t 2(q_2)_t+(u^{-\frac{2}{3}})_t (q_3)_t]\vol_0.
\end{split}
\end{equation*}
The final formula is obtained, by inserting the derivative of $u$ \eqref{uq3}. 
\end{proof}

\begin{lem}\label{psi vphitt}
$\psi\wedge \vphi_{tt}
=\frac{1}{2}\eps_{jkl}\lambda^{ij}d\theta^{kl3}\wedge \omega_i\wedge  d\alpha_t=\lambda^{i3} \omega_i\wedge  d\alpha_t  \wedge d\theta^{123}.$
\end{lem}

\subsubsection{$L^2$ Geodesic in local coordinates}
In order to expand the second part of the second variation of the volume, we need to use the geodesic equation in \corref{geodesic HS cor} and the following lemmas.
\begin{prop}\label{I pointwise}
Let $W=W^{j_1} \partial x_{j_1}$ be a vector field. Then $A^\alpha_3 ({\alpha^\sharp}  ,W)\vol_\vphi$
\begin{align*}
=W^{j_1}\alpha^{j_2} \{B_{i_1i_2i_3i_4,j_1j_2} \wedge dx^{i_1i_2i_3i_4}
+[(\om_1)_{j_2j_1} (q_2)_t
+    (\om_2)_{j_2j_1}(q_1)_t]\vol_0\}.
\end{align*}
\end{prop}
\begin{proof}
We write $\alpha=\alpha_i dx^i$, then
$
d\alpha=\frac{1}{2}\alpha_{ji} dx^{ji},\, \alpha_{ji}=\alpha_{i,j}- \alpha_{j,i} 
$ and $\alpha^{\sharp}=\alpha_k g^{kj} \partial x_j$.
Then the contraction
$$
i_{\alpha^\sharp}d\alpha= \alpha^{j}  \alpha_{ji} dx^i,\quad
 i_{\alpha^\sharp} \omega_1= \alpha^{j} (\om_1)_{ji} dx^i.
$$

And the contractions regarding to $W$ are
$$i_{W} d\alpha=  W^{k_4}\alpha_{k_4 i} \, dx^{i },\quad
i_{W} \omega_1=  W^{k_4}(\om_1)_{k_4 i} \, dx^{i }.
$$ Thus we get the 1st and 2nd term in $I$:
\begin{equation*}
\begin{split}
&i_{\alpha^\sharp} d\alpha\wedge i_{W} \omega_1
=\alpha^{j}  \alpha_{ji_1} W^{k_4}(\om_1)_{k_4 i_2} \, dx^{i_1i_2 },\\
&i_{W} d\alpha\wedge i_{\alpha^\sharp} \omega_1
=W^{k_4}\alpha_{ k_4i_1 } \alpha^{j} (\om_1)_{j i_2 }  \, dx^{i_1 i_2}.
\end{split} 
\end{equation*}

Moreover, we apply \lemref{ixalpha} to expand the 3rd term in $I$,
$$
i_{{\alpha^\sharp},{W}} \Om_3
=i_{W}\om_1\wedge i_{{\alpha^\sharp}}\om_2
+\frac{1}{2}(i_{W}i_{{\alpha^\sharp}}\om_1) \om_2
+\frac{1}{2}(i_{W}i_{{\alpha^\sharp}}\om_2 )\om_1 .
$$ The 1st one is 
$i_{W}\om_1\wedge i_{{\alpha^\sharp}}\om_2
=W^{k_4}(\om_1)_{k_4 i_1} \,\alpha^{j}  (\om_2)_{ji_2} dx^{i_1i_2}
$ and the other contractions are
\begin{align*}
i_{W}i_{{\alpha^\sharp}}\om_1
=  W^{k_4} \alpha^{j} (\om_1)_{jk_4}  , \quad
i_{W}i_{{\alpha^\sharp}}\om_2
=W^{k_4}  \alpha^{j}  (\om_2)_{jk_4}  .
\end{align*}
Then we use \lemref{omega d alpha} to further simplify the expression of $d\alpha\wedge\om_i$, $i=1,2$, and $d\alpha\wedge i_{{\alpha^\sharp},{W}} \Om_3$
\begin{align*}
= W^{k_4}\alpha^{j}\{(\om_1)_{k_4 i_1} \,  (\om_2)_{ji_2} d\alpha\wedge dx^{i_1i_2}
+[   (\om_1)_{jk_4} (q_2)_t
+   (\om_2)_{jk_4}(q_1)_t]\vol_0\}
.
\end{align*} 
Combining them together and and relabeling $j=j_2$, $k_4=j_1$, we have the final formula of $A_3^\alpha$.
\end{proof}

\begin{lem}\label{contraction ast}
Let $W$ be a vector field and $\gamma$ be a $p$-form. It holds
\begin{align*}
 \ast (W^\flat \wedge \ast\gamma)= (-1)^{(n-p)(p-1)} |g| \cdot  i_W \tilde\gamma,
\end{align*}
where the dual $1$-form
$
   W^\flat = g_{ij} W^i \, dx^j
$, 
$\epsilon_{\ast}$ is the Levi-Civita symbol,
\[
\tilde \gamma:=
\frac{  \gamma^{k_1 \cdots k_p}  g^{j_{1}l_2}\cdots g^{j_{n-p} l_{n-p+1}}  }{p! p!(n-p)!}  \, \epsilon_{k_1 \cdots k_p, j_{1} \cdots j_{n-p}}  \epsilon_{i, r_{1} \cdots r_{p-1}   ,l_2 \cdots l_{n-p+1} } \, dx^{i r_{1} \cdots r_{p-1}}.
\] 
\end{lem}
\begin{proof}
Under local coordinates $\{dx^i\}$, we write $\gamma= \frac{1}{p!}\gamma_{i_1 \dots i_p}dx^{i_1  \cdots i_p}$,
\[g = g_{ij} \, dx^i \otimes dx^j,\quad \mathrm{vol} = \sqrt{|g|} \, dx^1 \wedge \cdots \wedge dx^n.\]
   The contraction of a $p$-form $\alpha$ with the vector field $W = W^i \partial x_i$ is
\begin{align}\label{contraction local}
   i_W \gamma= \frac{1}{(p-1)!} W^j \gamma_{j i_2 \cdots i_p} \, dx^{i_2 \cdots i_p}
\end{align}
and the Hodge star operator $\ast$ acts on the $p$-form $\gamma$
\begin{align}\label{Hodge star local}
   \ast \gamma = \frac{\sqrt{|g|}\gamma_{i_1 \dots i_p} g^{i_1k_1}\cdots g^{i_pk_p} \, \epsilon_{k_1 \cdots k_p, j_{1} \cdots j_{n-p}} }{p!(n-p)!} \, dx^{j_{1} \cdots j_{n-p}}.
\end{align}

We compute the $(n-p+1)$-form $\beta:= W^\flat \wedge   \ast \gamma=\frac{\beta_{j  j_{1} \cdots j_{n-p}} dx^{ j  j_{1} \cdots j_{n-p}}}{(n-p+1)!}$, 
\[
\quad 
\beta_{j  j_{1} \cdots j_{n-p}}=(n-p+1)!\frac{\sqrt{|g|} W_j }{p!(n-p)!} \gamma^{k_1 \cdots k_p} \, \epsilon_{k_1 \cdots k_p, j_{1} \cdots j_{n-p}}.
   \]
Put it into \eqref{Hodge star local}, we get
\begin{align*}
 \ast \beta
 &=\frac{1}{(p-1)!}\frac{|g|  W_j \gamma^{k_1 \cdots k_p}  \epsilon_{k_1 \cdots k_p, j_{1} \cdots j_{n-p}} }{p!(n-p)! }  \, {\epsilon^{j j_1 \cdots j_{n-p}}}_{ r_{1} \cdots r_{p-1}}\, dx^{r_{1} \cdots r_{p-1}}.
\end{align*}
Thus we prove the lemma by using \eqref{contraction local}.
\end{proof}
\begin{prop}\label{ialphagamma}
Let the $3$-form $\gamma= i_{\alpha^\sharp} \omega_1\wedge \omega_2 $ in \lemref{contraction ast}. 
$$
\ast_4( \alpha\wedge \ast_4\gamma)=u^{\frac{2}{3}} i_{\alpha^\sharp}\tilde\gamma,\quad i_{\alpha^\sharp}\tilde\gamma=\frac{  \alpha^{i}\alpha^{j}  (\om_1)_{k_1j}(\om_2)_{k_2k_3}  }{4 }  \, {\epsilon^{k_1k_2 k_3l_2}}  \epsilon_{i r_{1}   r_{2}  l_2 } \, dx^{r_{1}  r_{2}}.$$
\end{prop}
\begin{proof}
We have
$
\ast_4( \alpha\wedge \ast_4\gamma)=|g|\cdot i_{\alpha^\sharp}\tilde\gamma,\quad |g|=u^{\frac{2}{3}}.
$
In local coordinates, $\gamma= \alpha^{j}  (\om_1)_{ij} dx^i\wedge \om_2= \frac{1}{2}\alpha^{j}  (\om_1)_{k_1j}(\om_2)_{k_2k_3} dx^{k_1k_2k_3}$.
Recall in \lemref{contraction ast} that $$
\tilde \gamma=
\frac{ \gamma_{k_1 k_2 k_3}    }{3!\times 3!}  \, {\epsilon^{k_1k_2 k_3 l_2}} \epsilon_{i r_{1}   r_{2}  l_2 } \, dx^{i r_{1}  r_{2}}
=\frac{  \alpha^{j}  (\om_1)_{k_1j}(\om_2)_{k_2k_3}  }{2\times 3! }  \, {\epsilon^{k_1k_2 k_3l_2}}  \epsilon_{i r_{1}   r_{2}  l_2 } \, dx^{i r_{1}  r_{2}}.
$$
Thus the formula is proved.
\end{proof}

\begin{prop}\label{dialphagamma}
$d i_{\alpha^\sharp}\tilde\gamma=\delta^{k_1 k_2 k_3}_{j_1  i_{1} i_{2}}  dx^{k i_{1}  i_{2}}$ multiplied with
$$
\big\{
\frac{1}{4}{u^{-\frac{2}{3}} [\alpha^{j_1}}  \alpha^{j_2}  (\om_1)_{k_1j_2} (\om_2)_{k_2k_3} ]_k
 -\frac{1}{6} u^{-\frac{5}{3}} u_{k} {\alpha^{j_1}}  \alpha^{j_2}  (\om_1)_{k_1j_2} (\om_2)_{k_2k_3} 
  \big\}.
$$

\end{prop}
\begin{proof}
From \propref{ialphagamma}, the exterior derivative of $i_{\alpha^\sharp}\tilde\gamma$ becomes
$$
d i_{\alpha^\sharp}\tilde\gamma=\frac{1}{4}
[
{\alpha^{j_1}}  \alpha^{j_2}  (\om_1)_{k_1j_2} (\om_2)_{k_2k_3}  \epsilon^{k_1 k_2k_3 k_4 }]_k { \epsilon}_{j_1 k_4  i_{1} i_{2}}  dx^{k i_{1}  i_{2}}.
$$
The determinant of the inverse metric tensor gives that
\[ \epsilon^{k_1 k_2k_3 k_4 }=g^{k_1 l_1} g^{k_2 l_2}g^{k_3 l_3}g^{k_4 l_4}\epsilon_{l_1 l_2 l_3 l_4 }=|g|^{-1}\epsilon_{k_1 k_2 k_3 k_4 }=u^{-\frac{2}{3}}\epsilon_{k_1 k_2 k_3 k_4 }.\]
Therefore its derivative is $(\epsilon^{k_1 k_2k_3 k_4 })_k =-\frac{2}{3}u^{-\frac{5}{3}} u_{k}\epsilon_{k_1 k_2 k_3 k_4 }$ and we further expand $d i_{\alpha^\sharp}\tilde\gamma$  to obtain the formula.
 \end{proof}

\subsubsection{Auxiliary functions and torsion form}
\begin{lem}\label{self-dual}
If $\om$ is a self-dual $2$-form, then
\begin{align*}
2\om_{j_1j_2}=\sqrt{|g|}\om_{i_1i_2} {\epsilon^{i_1i_2}}_{j_1j_2}.
\end{align*}
\end{lem}
\begin{defn}
We define a $3$-form 
$$\om_u:= d(u^{-\frac{2}{3}} \omega_3)=-\frac{2}{3}u^{-\frac{2}{3}}d\log u\wedge\omega_3= -\frac{1}{3}u^{-\frac{5}{3}} (\omega_{3})_{ i_1i_2} \,u_{i_3}  \, dx^{i_1i_2i_3}.$$
We also define the auxiliary function $\tilde u^{k_4}:= -\frac{1}{3} u^{\frac{-4}{3} }  u_{k_3}  (\omega_{3})_{ k_1k_2}  \epsilon^{k_1 k_2k_3 k_4}$.
\end{defn}
\begin{lem}\label{omega u}
We apply
 \lemref{self-dual} to see that
$$
\tilde u^{k_4}=-\frac{2}{3} u^{\frac{-5}{3} }  u_{k_3}   (\omega_{3})^{ k_3k_4},
\quad 
 \ast_4\om_u =   \tilde u_j dx^{j},
 \quad 
 (\ast_4 \om_u)^\sharp 
=  \tilde u^k\, \p x_k.
$$
\end{lem}


\subsubsection{$i=3$ in \lemref{psi vphitt}}
\begin{lem}\label{psi vphitt 123} 
The third component in \lemref{psi vphitt} is
\begin{equation*}
\begin{split}&6 \int_M u^{-\frac{2}{3}}  \omega_3\wedge  d\alpha_t \wedge d\theta^{123}
=I+II+III,\\
&
I:=-2\int_M \om_u \wedge  i_{\alpha^\sharp} A^\alpha_3 \wedge d\theta^{123},\\
&
II:=2 \int_M u^{\frac{2}{3}} \om_u \wedge  \delta[\alpha\wedge \ast_4(u^{-\frac{2}{3}} i_{\alpha^\sharp} \omega_1\wedge \omega_2)]\wedge d\theta^{123},\\
&
III:=-\int_M u^{\frac{2}{3}}\om_u \wedge  \delta[\ast_4(u^{-\frac{2}{3}} i_{\alpha^\sharp}\om_1\wedge i_{\alpha^\sharp}\om_2 )] \wedge d\theta^{123}.
\end{split}
\end{equation*}
\end{lem}
\begin{proof}
When $i=3$, $\lambda^{33}=u^{-\frac{2}{3}}$. We plug the geodesic equation \corref{geodesic HS cor} into the expression $6 \int_M d u^{-\frac{2}{3}}\wedge d\theta^{123}\wedge \omega_3\wedge  \alpha_t
$ above.
\end{proof}

\begin{prop}
$I=\frac{2}{3} \int_M   \,
u^{\frac{-2}{3} }  (\log u)^{j_3}   {(\omega_{3})_{j_3}}^{j_1}\,
\alpha^{j_2} \,$
\begin{equation*}
\begin{split}
\times \big\{&  B_{i_1i_2i_3i_4,j_1j_2} \wedge dx^{i_1i_2i_3i_4}
+    2 [(\om_1)_{j_2 j_1}   (q_2)_t 
+ (\om_2)_{j_2 j_1}   (q_1)_t] \vol_0 \,
\big\}\wedge d\theta^{123}.
\end{split} 
\end{equation*}
\end{prop}
\begin{proof}
We rewrite the first term $I$ with the help of the formula of $A^\alpha_3$ in \lemref{pg},
$$\om_u \wedge  i_{\alpha^\sharp} A^\alpha_3 \wedge d\theta^{123}=g_4( i_{\alpha^\sharp} A^\alpha_3  ,\ast_4 \om_u)\vol_\vphi
=  A^\alpha_3 ({\alpha^\sharp}  ,{(\ast_4 \om_u)^\sharp})\vol_\vphi.$$
We further expand each terms in
\begin{equation*}
\begin{split}
I&=-2\int_M 
[i_{\alpha^\sharp} d\alpha\wedge i_{(\ast_4 \om_u)^\sharp} \omega_1\wedge \omega_2
+ i_{(\ast_4 \om_u)^\sharp} d\alpha\wedge i_{\alpha^\sharp} \omega_1\wedge \omega_2
+d\alpha\wedge  i_{{\alpha^\sharp},{(\ast_4 \om_u)^\sharp}} \Om_3]\wedge d\theta^{123}.
\end{split} 
\end{equation*}
Thus $I$ is obtained by letting $W=(\ast_4 \om_u)^\sharp$ in \propref{I pointwise}.
\end{proof}

\begin{lem}
$II=-\int_M
 dx^{i_3 i_{1}  i_{2} i_{4} } \wedge d\theta^{123} \, \delta_{j_1   i_{1} i_{2}}^{k_1 k_2k_3  } $
\begin{equation*}
\begin{split}
&\times
\{
 -\frac{2}{9} u^{\frac{-8}{3} }  u^{j_3} u_{i_3}  
 \alpha^{j_1} \alpha^{j_2}  (\om_1)_{k_1j_2} {(\om_2)_{k_2k_3}}(\omega_{3})_{ j_3i_4} 
 \\
&+\frac{1}{3} u^{\frac{-5}{3} }  u^{j_3}   (\omega_{3})_{ j_3i_4}\, [
{(\alpha^{j_1}}  \alpha^{j_2})_{i_3}  (\om_1)_{k_1j_2} (\om_2)_{k_2k_3}  +
\alpha^{j_1}\alpha^{j_2} ( {(\om_1)_{k_1j_2}} (\om_2)_{k_2k_3} )_{i_3}] 
 \}.
\end{split}
\end{equation*}
\end{lem}
\begin{proof}
From \lemref{psi vphitt 123}, the second term $II$ is  
\begin{equation*}
\begin{split}
=2 \int_Mg_4( u^{\frac{2}{3}} \om_u, d \{u^{-\frac{2}{3}}\ast_4[ \alpha\wedge \ast_4( i_{\alpha^\sharp} \omega_1\wedge \omega_2 )]\}) \vol_\vphi.
\end{split}
\end{equation*}
According to \propref{ialphagamma}, the integral $II$ is reduced to
\begin{align*}
II=-\frac{4}{3}  \int_Mg_4( d\log u\wedge\omega_3, d i_{\alpha^\sharp}\tilde\gamma ) \vol_\vphi
=2 \int_Md i_{\alpha^\sharp}\tilde\gamma \wedge \ast_4(u^{\frac{2}{3}} \om_u).
\end{align*}
Due to \lemref{omega u}, we have
$$
 \ast_4 (  u^{\frac{2}{3}} \om_u)=-\frac{2}{3}\ast_4 (d\log u\wedge\omega_3)
    =u^{\frac{2}{3}}\,\tilde u_{i_4} \, dx^{i_4}
    =-\frac{2}{3} u^{-1 }  u^{j_3}   (\omega_{3})_{ j_3i_4} \, dx^{i_4}.
$$ 
The formula of $II$ follows from inserting these identities and \propref{dialphagamma} for $d i_{\alpha^\sharp}\tilde\gamma $ into the integral above
and relabel $k=i_3$.
\end{proof}

\begin{lem}$III=- \int_M  dx^{i_3 i_1 i_2i_4 }\wedge d\theta^{123}$
\begin{equation*}
\begin{split}
&\times
\big\{\frac{4}{9} u^{-\frac{8}{3}}  u^{j_3}  u_{i_3}  \alpha^{j_1}  \alpha^{j_2} 
 (\om_1)_{i_1j_1} (\om_2)_{i_2j_2} (\omega_{3})_{ j_3i_4} \\
&-\frac{2}{3}u^{-\frac{5}{3}} u^{j_3}   (\omega_{3})_{ j_3i_4} 
[
 (\alpha^{j_1}  \alpha^{j_2} )_{i_3} (\om_1)_{i_1j_1} (\om_2)_{i_2j_2} 
+ \alpha^{j_1}  \alpha^{j_2} ( (\om_1)_{i_1j_1} (\om_2)_{i_2j_2} )_{i_3}] \big  \}.
\end{split}
\end{equation*}
\end{lem}
\begin{proof}
The third term 
$
III=-\int_M d [u^{-\frac{2}{3}}  i_{\alpha^\sharp}\om_1\wedge i_{\alpha^\sharp}\om_2]\wedge\ast_4 (u^{\frac{2}{3}}\om_u) .
$
But 
\begin{equation*}
\begin{split}
&d [u^{-\frac{2}{3}}  i_{\alpha^\sharp}\om_1\wedge i_{\alpha^\sharp}\om_2]=dx^{i_3 i_1i_2}\\
&\times\{ -\frac{2}{3} u^{-\frac{5}{3}}  u_{i_3}\alpha^{j_1}  \alpha^{j_2}  (\om_1)_{i_1j_1} (\om_2)_{i_2j_2} +u^{-\frac{2}{3}} [\alpha^{j_1}  \alpha^{j_2}  (\om_1)_{i_1j_1} (\om_2)_{i_2j_2} ]_{i_3}\} ,
\end{split}
\end{equation*}
which infers the expression of $III$.
\end{proof}
\begin{prop}
  $II+III=-\int_M
  dx^{i_3 i_{1}  i_{2} i_{4} } \wedge d\theta^{123} \,$
\begin{equation*}
\begin{split}
 &\times \frac{1}{3} (\log u)^{j_3}   (\omega_{3})_{ j_3i_4}\,\times
(u^{\frac{-2}{3} }
\alpha^{j_1}\alpha^{j_2}  A_{i_1i_2,j_1j_2})_{i_3}.
\end{split}
\end{equation*}
\end{prop}
\begin{rem}
When $i=1,2$, $\lambda^{i3}=-q_i u^{-\frac{2}{3}}$. So, $\om_u$ should be replaced by $-d(q_iu^{-\frac{2}{3}} \omega_i)$.
The expression for $6 \int_M u^{-\frac{2}{3}}  (-q_1\omega_1-q_2\omega_2)\wedge  d\alpha_t \wedge d\theta^{123}$
is obtained similarly.
\end{rem}


\subsection{Flow and weighted volumes}
Form the triple \eqref{mix HS}, the $G_2$ flow \eqref{Laplacian flow} becomes $$\vphi_t={\Pr}_3 d\tau,$$ where the torsion is given in \corref{torsion form and d torsion} and
the notation $\Pr_i$ means the projection to the $i$-th component of the triple.

The flow of the symplectic form \eqref{Laplacian flow form} is reduced to 
	 \begin{equation}\label{Laplacian flow form diagonal}
	\begin{split}
	(\om_3)_t=-d\tau_3,\quad \tau_3=\ast_4 (a^p\wedge\omega_p),
		\end{split}
	\end{equation} where $a^p=\sum_{j} \lambda_{3j} \, d\lambda^{pj}$ and 
\[
\begin{bmatrix}
a^1 \\
a^2 \\
a^3
\end{bmatrix}
= 
\frac{du}{u}
\begin{bmatrix}
0 \\
0 \\
-\frac{2}{3} \\
\end{bmatrix}
-\frac{1}{u}
\begin{bmatrix}
q_3 -q_2^2  & q_1 q_2 & -q_1 \\
 q_1 q_2 & q_3-q_1^2  & -q_2 \\
 q_1 & q_2 & 0
\end{bmatrix}
\begin{bmatrix}
dq_1 \\
dq_2 \\
dq_3
\end{bmatrix}
.
\]
\begin{lem}
Differentiating \eqref{Qij}, we see that
	 \begin{equation}\label{Qij pt}
	\begin{split}
(\log u)_t=q^{ij}\p_t q_{ij}= q^{ij}\frac{d\alpha_i\wedge \omega_j}{\vol_0}= q^{3j}\frac{d\alpha\wedge \omega_j}{\vol_0}.
	\end{split}
	\end{equation}
\end{lem}
\begin{cor}\label{diagonal flow}
If $\mathcal Q$ is diagonal, then $q_1=q_2=0$, $q_3=u$, $\lambda_{33}=u^{\frac{2}{3}}$, $a^3= -\frac{2}{3}d\log u$ and
$(\om_3)_t=\frac{2}{3} d \ast_4 (d\log u \wedge\omega_3).$ Furthermore,
$$u_t= \frac{2d [\ast_4 (d\log u \wedge\omega_3)\wedge\om_3]}{3\vol_0}.$$
\end{cor}
\begin{proof}
From \eqref{Laplacian flow form diagonal}, we have the evolution of $\om_3$. The evolution of $u$ is obtained from \eqref{Qij pt}
$
(\log u)_t=u^{-1}\frac{\p_t \omega_3\wedge \omega_3}{\vol_0}.
$
\end{proof}


\begin{defn}\label{weighted volume functional 3}
We denote the triple $\underline \vphi:=(\vphi_1,\vphi_2,\vphi_3)$ and extend the $\chi$ volume functional in Definition \ref{weighted volume functional 1} for the triples \eqref{HS vector},	
	\begin{align}\label{3 weighted volume functional}
	\Vol_{\underline \chi}(\underline \vphi):= \sum_i \Vol_{\chi_i}(\vphi_i)=\sum_i \int_M \chi_i(u_i) \vol_E,
	\end{align}
	where $\{\chi_i, i=1,2,3\}$ are nonnegative concave functions.
\end{defn}	
The variations of the $\underline\chi$ volume functional along the family $\underline\vphi(t,s)$ are also the combination of \eqref{eq:2nd_vol chi general} for each $\vphi_i$,
\begin{equation*}
\begin{split}
	\p_t\Vol_{\underline\chi}=\sum_i\p_t\Vol_{\chi_i}(\vphi_i),\quad \p^2_{ts}\Vol_{\underline\chi}=\sum_i  \p^2_{ts}\Vol_{\chi_i}(\vphi_i) .
\end{split}
\end{equation*}
The gradient flows for the triples $\underline\om_i$ are defined in the same way.
Finally, a natural upper bound is obtained from Jensen's inequality.
\begin{lem}
$\Vol_{\underline \chi}(\underline \vphi)\leq  \sum_i \chi_i( \Vol(\vphi_i))$.
\end{lem}

\section{Torus fibration}\label{Torus fibration}
We choose $X^4:=\mathbb{T}^4$ and the local coordinates $\{x^0, x^1, x^2, x^3\}$ on $\mathbb{T}^4$. For simplicity, we use the notations 
$d\theta^{\bar{i}}:= d\theta^{i+1}\wedge d\theta^{i+2}$, $dx^{\bar{i}}:= dx^{i+1}\wedge dx^{i+2}$.
A standard hyper-K\"ahler structure on $\mathbb{T}^4$ is given by the triple
	$$\underline{\omega}^K=(\omega^K_1,\omega^K_2,\omega^K_3): \omega^K_i:= dx^{0i}+dx^{\bar{i}},\quad i=1,2,3$$
which determines a flat torsion-free $G_2$ structure $\vphi^K$ on the torus fibration $M=\mathbb{T}^3\times \mathbb{T}^4$.

Let the symplectic structure $\underline\om_3$ \eqref{mix HS} have the special form
	\begin{align}\label{hypers}
		\omega_1:= dx^{01}+dx^{23},\quad \omega_2:= dx^{02}+dx^{31},\quad \om_3:= u \, dx^{03}+dx^{12},
	\end{align}
where the function $u$ is strictly positive.
Because we require that $\om_3$ stays in the same cohomology of $\om_3^K$, $(u-1)dx^{03}=d\eta$ is $d$-exact. Since the integrability condition $d[(u-1)dx^{03}]=dd\eta=0$, we can see that there is a constraint on $u$, i.e. $\p_1 u=\p_2 u=0$.
\begin{lem}
$u$ only depends on $x_0$ and $x_3$, $\int u dx^{03}=\int dx^{03}=1$.
\end{lem}

We compute directly from the previous section.
We set $w:= u^{\frac{1}{3}}$ and rewrite the hyper-symplectic structure \eqref{hypers}
	\begin{align}\label{vvi}
	 \omega_i=wv_idx^{0i}+dx^{\bar{i}},\quad  v_1=v_2=w^{-1},\quad v_3=w^2.
	 \end{align}
	 \begin{lem}\label{Q diagonal}
The 3-by-3 matrix $\mathcal Q$ is diagonal
$$(q_{ij})=diag\{1,1,u\},\quad (\lambda_{ij})=diag\{v_1,v_2,v_3\}.$$
\end{lem}

The hyper-symplectic structure $\underline{\om}$ determines the closed $G_2$ structure
\begin{align}\label{vphi}
\vphi=d\theta^{123}-\sum_{1,2} d\theta^i\wedge dx^{0i}-u d\theta^3\wedge dx^{03}-\sum_{1,2,3} d\theta^i\wedge dx^{\bar i}.
	\end{align}

The volume element on $M$ is 
		$
		\vol_\vphi= w \cdot dx^{0123}\wedge d\theta^{123}$.
The $G_2$ metric \eqref{g3g4} associated to $\vphi$ is diagonal
		\begin{align}\label{Inverse metric}
			 g_4=diag(u^{\frac{2}{3}},u^{\frac{-1}{3}},u^{\frac{-1}{3}},u^{\frac{2}{3}}),
			\quad g_3=diag(u^{\frac{-1}{3}},u^{\frac{-1}{3}},u^{\frac{2}{3}}).
		\end{align} 
The Hodge star actions: 
\begin{equation}\label{riemannian HS 1 hodge}
\begin{split}
&*_3 d\theta^i=v_i^{-1}d\theta^{\bar{i}},\quad *_4dx^{0}=w^{-1}dx^{123},\quad *_4dx^{i}=-wv_i^{-1}dx^{0\bar{i}},\\
&v_iw^{-1} dx^{i}= *_4 dx^{0\bar{i}},\quad
*_4dx^{0i}=(wv_i)^{-1}dx^{\bar{i}},\quad wv_i dx^{0i}= *_4 dx^{\bar{i}},
\end{split}
\end{equation}
which give us the dual $4$-form
		\begin{align*}
			\psi=u^{\frac{1}{3}} dx^{0123}
			-\sum_{1,2} u^{\frac{1}{3}}d\theta^{\bar{i}} \wedge dx^{\bar{i}} 
			-u^{\frac{-2}{3}}d\theta^{\bar{3}} \wedge dx^{\bar{3}}
			-u^{\frac{1}{3}} \sum_i d\theta^{\bar{i}} \wedge dx^{0i} .
		\end{align*}

\begin{lem}
The torsion form $\tau=\sum_i  d\theta^i\wedge \tau_i$, $\tau_i=\ast_4(d\log v_i\wedge \om_i)$. Precisely, 		\begin{equation}\label{eq:tau}
		\begin{split}
&\tau_1=-\frac{1}{3}u^{-\frac{5}{3}}(u_0dx^1-u_3dx^2),\quad\tau_2=-\frac{1}{3}u^{-\frac{5}{3}}(u_0dx^2+u_3dx^1),\\
& \tau_3=\frac{2}{3}u^{-\frac{2}{3}}(u_0dx^3-u_3dx^0).
			\end{split}
		\end{equation}
\end{lem}

\subsubsection{Tangent vectors}
	\begin{defn}\label{defn:space 1}
		We define $\mathcal{M}_1$ to be the space of the closed $G_2$ structures of the form \eqref{vphi} in the cohomology class $[\vphi^K]$ on $\mathbb{T}^7$.
	\end{defn}
	The definition of $\vphi$ \eqref{vphi} implies the tangent vector 
	\begin{align}\label{vphit}
		X=-\underline{f},\quad \underline{f}:= f d\theta^3 \wedge dx^{03} ,\quad f=\p_tu,\quad \int f dx^{03} =0.
	\end{align}	

The potential form $\mathbf u$ has the form \eqref{vphit general 1}. We need to solve the following Hodge-Laplacian equations
\begin{align}\label{dalpha f}
d \alpha_c=fdx^{03},\quad \alpha_c=-(\Lambda_3+ \delta   u_3).
\end{align}
The unknown $3$-form $u_3$ has 6 components and the identity above gives us 6 equations as well.

\subsection{Hessian of volume}
 We write $\alpha=\alpha_i dx^i$ in local coordinates. 
\begin{prop}\label{geodesic concavity body}
Suppose the $1$-form $\alpha$ has the normalisation 
  \begin{align}\label{alpha normalisation}
 \alpha_0=\alpha_1=\alpha_2=0,\quad\p_0\alpha_3=f.
 \end{align}
 Then, the volume is geodecically concave in \thmref{Vtt diagonal}:
 $$
9\Vol _{tt}
=-2\int_M[u^{-1}f-u^{-2}u_0\alpha_3]^2\vol_\vphi 
=-2\int_M[\p_0(\frac{a_3}{u})]^2\vol_\vphi .
$$
\end{prop}
\begin{proof}
Since $\mathcal Q$ is diagonal by \lemref{Q diagonal}, we further simplify \thmref{Vtt diagonal}, 
$
18\Vol _{tt}=\int_M -4\, f^2 \, u^{-2} \vol_\vphi+I_1+I_2,
$
in which,
\begin{align*}
I_1
 = 4 \,f\,u_0\,\alpha _{3}\,u^{-3}\,\vol_\vphi
,\quad I_2=[ 
-4\,u_0^2\,\alpha _{3}^2
+4\,u_0\,\alpha _{3}\,u\,\p_0\alpha _{3}
]\,u^{-4} \vol_\vphi.
\end{align*}
Therefore, geodesic concavity of the volume functional is obtained.
\end{proof}

\begin{lem}
Let $ \boldsymbol{\alpha}:=\alpha_3  d\theta^3 dx^3$ and $\mathcal V:=\frac{1}{3}\frac{\a_3}{u}\p_{x_0}$. Then \begin{itemize}
\item
$\pi_7^2\boldsymbol{\alpha}=\frac{1}{3}[\frac{\a_3}{u} \cdot d\theta^1dx^1+\frac{\a_3}{u}\cdot d\theta^2dx^2+\a_3\cdot d\theta^3dx^3]
=i_{\mathcal V}\vphi$.

\item 
$\pi_{14}^2\boldsymbol{\alpha}
=\frac{1}{3}[-\frac{\a_3}{u}\cdot d\theta^1dx^1-\frac{\a_3}{u}\cdot d\theta^2dx^2+2\a_3 \cdot d\theta^3dx^3],\quad \delta \pi_{14}^2\boldsymbol{\alpha}=0$.
	
\end{itemize}
\end{lem}
	
\begin{defn}	
Recall that an orientation-preserving diffeomorphism $f:S^1\rightarrow S^1$ is a lift $\mathbb R\rightarrow \mathbb R$ satisfies that $F$ is smooth, periodic $F(x+1)=F(x)+1$ and strictly increasing.
\end{defn}

\begin{lem}\label{diffeomorphism PDE solution}
Assume that the boundary values satisfy the consistency condition $\tilde u_0(x_0+C(x_3),x_3)=\tilde  u_1(x_0,x_3)$ for some function $C(y)$.
Then, the boundary value problem of the partial differential equation of $u(x,t)$ for $x\in S^1$ and $t\in [0,1]$,
\begin{equation}\label{diffeomorphism PDE}
\begin{split}
K(u_t)=u_t-3 \mathcal V (u)=0,\quad
u(x,0)=\tilde u_0,\quad u(x,1)=\tilde u_1
			\end{split}
\end{equation}
		has a smooth solution. Moreover, the boundary values are connected by a path of orientation-preserving diffeomorphisms on $S^1$.
\end{lem}
\begin{proof}
Since $\p_0 \alpha_3=u_t$ Definition \ref{operator L and Q}, we rewrite \eqref{diffeomorphism PDE} as
		$[\log (\alpha_3 u^{-1})]_0=0.$
Therefore, there exists a function $C_1(x_3,t)$ such that
		$\alpha_3=e^{C_1}u.$

Taking $x_0$ derivatives, we obtain the equation
		$u_t=e^{C_1}u_0.$
The general solution is
		$$u(x_0,x_3,t)=H(x_0+C_2,x_3),\quad 
		\frac{d}{d t}C_2=e^{C_1}.$$ Here, $H $ is an arbitrary positive smooth function. Then, by applying the boundary condition, we have the solution to \eqref{diffeomorphism PDE}
		$$u(x_0,x_3,t)=\tilde u_0(x_0+C_3(t,x_3),x_3),\quad C_3(t,x_3):=\int_0^t e^{C_1(s,x_3)}ds,$$
where $C_3(1,x_3)=C(x_3)$.	
We further see that the rigid rotation $F(x)=x+\int_0^t e^{C_1(s,x_3)}ds$ is an orientation-preserving diffeomorphism.
\end{proof}

\subsection{Cohomogeneity one hyper-symplectic structures}\label{HS and OT}
From now on, we assume that the function $u$ depends only on $x_0$ and consider normalised functions $f \in C_0^\infty(S^1)$.  
We identify the circle $S^1$ with the quotient $\mathbb{R}/2\pi\mathbb{Z}$ and denote by $C_0^\infty(S^1)$ the space of smooth functions on $S^1$ with zero mean, i.e., satisfying $\int_{S^1} f \, d\theta = 0$.

\begin{prop}\label{Lambda_3 +delta Om_3 Q}The canonical $1$-form in Definition \ref{vphit general alphai} becomes
$$ \alpha_c= Q(f)dx^{3},\quad \p_0 [Q(f)]=f.$$
Moreover, the normalisation \eqref{alpha normalisation} is satisfied.
\end{prop}
This follows immediately from the next proposition.
\subsubsection{Solving the potential form}
\begin{prop}\label{lem:Green}
The potential form and the canonical form become
\begin{equation*}
	\begin{split}
\mathbf u= u^{\frac{5}{3}}  \big\{L[Q_u(f)]+C_2\big\}d\theta^3\wedge dx^{03} ,\quad  \boldsymbol{\alpha_c}=\delta \mathbf u=Q(f) d\theta^3\wedge dx^3,
		\end{split}
	\end{equation*}
	where we define the constant $C_2:= -\big\{\int_{S^1}u^{-\frac{5}{3}} L[Q_u(f)] \big\} \cdot [\int_{S^1}u^{-\frac{5}{3}}]^{-1}$ to ensure that the potential $\mathbf u$ is a $d$-exact $3$-form.
\end{prop}
Assume that the $2$-form $u_3$ in the potential form $\mathbf u$ \eqref{vphit general 1} has the form $u_3= d \mu$, the $1$-form $\mu=\mu_i dx^i$, and $$u_3=\sum_{i=1,2,3} 
\beta_{0i}dx^0\wedge dx^i
+\beta_{12}dx^1\wedge dx^2
+\beta_{23}dx^2\wedge dx^3
+\beta_{31}dx^3\wedge dx^1
.$$ 
In which,
$\beta_{0i}= \p_0 \mu_i - \p_i \mu_0 ,\quad \beta_{ij}=\p_i \mu_j - \p_j \mu_i.$

Assume that the four coefficients $\mu_i$ are unknown functions depending only on $x_0 \in S^1$. 
Throughout this section, we denote derivatives with respect to $x_0$ by either $(\cdot)'$ or $(\cdot)_0$. 

First, we observe that only three of the coefficients $\beta_{0i}$ are non-vanishing:
\begin{align*}
\beta_{0i}= \p_0 \mu_i,\quad \beta_{ij}=0,\quad   i,j=1,2,3.
\end{align*}

\begin{lem}
We apply \lemref{Laplacian general} to $u_3$ and obtain
\begin{align*}
&\alpha_0=0 ,\quad 
\alpha_i=-(u^{-\frac{2}{3}}\beta_{0i})_0
, i=1,2,
 \quad
\alpha_3
=-u(u^{-\frac{5}{3}}\beta_{03})_0 . 
\end{align*}
\end{lem}

Second, the Hodge-Laplacian equation \eqref{dalpha f} reduces to the following system of ordinary differential equations:
\begin{lem}\label{dalpha f u3}
$
[u^{-\frac{2}{3}} \beta_{0i}]_{00}=0, i=1,2,
 \quad
[-u(u^{-\frac{5}{3}} \beta_{03})_0 ]_0=f.	
$
\end{lem}

For the first two ODEs, we observe that $\beta_{0i} = c_i u^{2/3}$ for some constants $c_i$. However, since $\beta_{0i} \in C_0^\infty(S^1)$ must have zero mean, this necessarily implies $c_i = 0$ for $i=1,2$. We therefore conclude that
\[
\alpha_1 = \alpha_2 = 0.
\]

The remaining coefficient $\beta_{03}$ can be determined by integrating the second ODE in Lemma~\ref{dalpha f u3} twice, subject to appropriate normalisation ensuring $\beta_{03} \in C_0^\infty(S^1)$.

\begin{defn}[Operators $L$ and $Q$]\label{operator L and Q}
For any $f\in C_0^\infty (S^1)$, we write
\begin{itemize}
\item $L(f):= F-\int_{S^1} F ,\quad F(x):= \int_0^x f(\tau)d\tau$,
\item $Q(f):= F+C_1(f),\quad C_1(f):=  -\big[\int_{S^1}u^{-1}F  \big]\big[ \int_{S^1} u^{-1} \big]^{-1}$ .
\end{itemize}
They are both particular solutions to the linear first-order  ordinary differential equation $[L(f)]'=f$ on $S^1$. Also, both $L(f)$ and $Q_u(f):= u^{-1} Q(f)$ lie in $C_0^\infty(S^1)$.
\end{defn}

\begin{proof}[Proof of \propref{lem:Green}]
 Applying Definition \ref{operator L and Q} twice to integrate the second equation \lemref{dalpha f u3}, we have $(\beta_{03}u^{-\frac{5}{3}})' u=- Q(f) $, then	
	$$-\beta_{03}u^{-\frac{5}{3}}=L[Q_u(f)]+C_2$$ such that $\beta_{03}$ lies in $ C_0^\infty(S^1)$. The expressions of $\mathbf u$ and $\delta \mathbf u$ follow from \lemref{Laplacian general}.
\end{proof}
	
\subsubsection{Dirichlet metric}
\begin{prop}Let $X:=-\underline{f}$, $Y:=-\underline{g}$ 
be vectors in $T_\varphi \mathcal{M}_1$. Then
\begin{equation}\label{eq:Dmetric}
		\begin{split}
			\mathcal G_\vphi(X,Y)
			&= \int_M u^{-1} Q (f)Q(g)
			=-\int_M L[Q_u(f)]  g  .
			\end{split}
		\end{equation}
\end{prop}
\begin{proof}
It follows from inserting \propref{lem:Green} in \eqref{gradient metric G}, or plugging \propref{Lambda_3 +delta Om_3 Q} into \lemref{metric mixed}. 
\end{proof}


\subsubsection{Canonical connection}
	\begin{prop}\label{defn:con}
		Let $\vphi(t)$ be a smooth path in $\mathcal{M}_1$ and $\vphi_t=X$. For any $Y$, the Levi-Civita connection is given by
		\begin{equation*}
		\begin{split}
		&D_X Y=Y_t+\mathbf  P(X,Y), \quad \mathbf  P(X,Y):= U' d\theta^3\wedge dx^{03},\\
			&U:=\frac{fQ(g)+Q(f)g}{2u}+ \frac{u}{2}\bigg[\frac{Q(f)Q(g)}{u^2}\bigg]'
			=\frac{fQ(g)+Q(f)g}{u}-\frac{Q(f)Q(g)u'}{u^2}.
		\end{split}
		\end{equation*}
	\end{prop}
	\begin{proof}
	It follows directly from inserting the next \lemref{A torus} into the connection for $1$-forms $\a,\b$ (Definition \ref{connection formula}). Then we use the relation $\mathbf  P(X,Y)=-d\theta^3\wedge d   P(\a,\b)$.
	\end{proof}

\begin{lem}\label{A torus}In Definition \ref{connection formula}, the mixed terms become
$$
  i_{\beta_c^\sharp} A^{\alpha_c}_3=3 f Q_u(g)dx^3,
\quad 
 i_{\alpha_c^\sharp} A^{\beta_c}_3=3 g Q_u(f)dx^3.
$$
All the following three terms $u^{\frac{2}{3}}\delta[\alpha_c\wedge \ast_4(u^{-\frac{2}{3}} i_{\beta_c^\sharp}\Om_3)]$,
$u^{\frac{2}{3}}\delta[\beta_c\wedge \ast_4(u^{-\frac{2}{3}} i_{\alpha_c^\sharp}\Om_3)]$,
$-u^{\frac{2}{3}}\delta[\ast_4(u^{-\frac{2}{3}}  i_{\alpha_c^\sharp,\beta_c^\sharp} \Om_3)]$
are equal to 
$$-u [Q_u(f)Q_u(g)]'dx^3.$$
\end{lem}	
\begin{proof}
The 1-form in Proposition \ref{Lambda_3 +delta Om_3 Q} is lifted to be a vector $\beta_c^\sharp= u^{-\frac{2}{3}}Q(g)\p x_{3}$.
By \lemref{pg}, $ i_{\beta_c^\sharp} A^{\alpha_c}_3=A^{\alpha_c}_3(\beta_c^\sharp,\p x_{3}) dx^3$ and $A^{\alpha_c}_3(\beta_c^\sharp,\p x_{3})$ 
	\begin{align*}
 = \frac{Q_u(g)}{\vol_0}[
i_{\p x_{3}} d\alpha_c\wedge i_{\p x_{3}}  \omega_1\wedge \omega_2
+ i_{\p x_{3}}  d\alpha_c\wedge i_{\p x_{3}}  \omega_1\wedge \omega_2
+d\alpha_c\wedge  i_{{\p x_{3}} ,{\p x_{3}} } \Om_3].
  	\end{align*}
	In which, by \eqref{hypers}, \eqref{dalpha f} and \lemref{ixalpha}, it holds that
	\begin{align*}
i_{\p x_{3}} d\alpha_c=-fdx^0,\quad i_{\p x_{3}}  \omega_1=-dx^2,\quad i_{\p x_{3}}  \omega_2=dx^1,\quad i_{{\p x_{3}} ,{\p x_{3}} } \Om_3=dx^{12}.
  	\end{align*}	
So, we have proved the first identity and the second one is got similarly.	
	
Also, from these identities, we get
$i_{\beta_c^\sharp}\Om_3=i_{\beta_c^\sharp}  \omega_1\wedge \omega_2=-u^{-\frac{2}{3}}Q(g)dx^{123}$ 	and $i_{\alpha_c^\sharp,\beta_c^\sharp} \Om_3=u^{-\frac{4}{3}}Q(f)Q(g)dx^{12}$. 
By \eqref{riemannian HS 1 hodge}, $\ast dx^{123}=-u^{\frac{1}{3}}dx^0$, $\ast dx^{30}=-u^{-1}dx^{12}$ and $\ast dx^{012}=u^{\frac{1}{3}}dx^3$. 
With all this identities,
\begin{align*}
&u^{\frac{2}{3}}\delta[\alpha_c\wedge \ast(u^{-\frac{2}{3}} i_{\beta_c^\sharp}\Om_3\wedge d\theta^{123})]\\
&
=-u^{\frac{2}{3}}\delta[\alpha_c\wedge \ast(u^{-\frac{4}{3}}Q(g)dx^{123} d\theta^{123})]
=u^{\frac{2}{3}}\delta[ u^{-1}Q(f)Q(g)dx^{30}]\\
&
=-u^{\frac{2}{3}}\ast d[ u^{-2}Q(f)Q(g)dx^{12}]
=-u [ u^{-2}Q(f)Q(g)]' dx^{3}
  	\end{align*}
	and the last two identities are obtained in the same way.
\end{proof}
\subsection{Canonical geodesic}	
The canonical geodesic equation $D_X X=0$ is read directly from Proposition \ref{defn:con}:
\begin{align}\label{1d geodesic}
	u_t=f,\quad f_t=\tilde U _x
,\quad \tilde U:=\frac{Q(f)f}{u}+ \frac{u}{2}\bigg[\frac{Q^2(f)}{u^2}\bigg]_x.
\end{align}
In which, we recall the notion of the canonical form in Definition \ref{operator L and Q},
\begin{align}\label{expression of Q(f)}
Q(f)=\int_0^x f(\tau)d\tau-\big[\int_{S^1}u^{-1}\int_0^x f(\tau)d\tau  \big]\big[ \int_{S^1} u^{-1} \big]^{-1}.
\end{align}

The function $u(x,t)$ is a family of nonnegative functions defined on the circle $S^1$. The parameter $t$ takes values in the unit interval $[0,1]$. 
In order to construct geodesic segment, we need to solve \eqref{1d geodesic} with smooth boundary values $$u(x,0)=u_0,\quad u(x,1)=u_1.$$

\begin{lem}\label{equations geodesic}
The canonical geodesic equation \ref{1d geodesic} is equivalent to
\begin{enumerate}
\item \label{geodesic 1}
the equation of the canonical form $Q(f)_{t}=\tilde U$, by putting \eqref{A torus} into \corref{geodesic HS cor}, 
\item \label{IB}
the inviscid Burgers' equation 
	\begin{align*}
	q_t+qq_x=0,\quad q:=-Q_u(f)=-Q(f)u^{-1};
	\end{align*}
\item \label{HJ}
the Hamilton-Jacobi equation 
\begin{align*}
\phi_t+\frac{|\phi_x|^2}{2} =B,\quad B:=\{[\phi_t+\frac{1}{2} |\phi_x|^2]\}(x=0),\quad q:=\phi_x.
\end{align*}
\end{enumerate}
\end{lem}
\begin{proof}
We use \eqref{geodesic 1} to compute $-q_t$ in \eqref{IB}
\begin{align*}
	&=[Q_u(f)]_t=[u^{-1}Q(f)]_t=-u^{-2}fQ(f)+u^{-1}Q(f)_{t} \\
	&=\frac{1}{2} [Q^2_u(f)]_x-A u^{-1}=qq_x-A u^{-1}.
\end{align*}

Putting $p=\phi_x$ in \eqref{IB}, we get
$
\p_x[\phi_t+\frac{1}{2} |\phi_x|^2]=A u^{-1}.
$ Then, integrating from $0$ to $x$, we have the equation \eqref{HJ} of $\phi$.
\end{proof}


\begin{rem}\label{optimal transport}
We could compare all these $G_2$ geometry findings to their counterparts in the theory of optimal transport.
If we interpret $u(t)$ as a smooth family of measures in the $2$-Wasserstein space, the velocity vector field $\nabla \phi$ is defined by the continuity equation
$$u_t=-\nabla (u \nabla \phi)$$ and the gradient metric is defined to be $$\int  u \nabla\phi_1 \nabla\phi_2.$$ It is exactly the same to the Dirichlet metric \eqref{eq:Dmetric}, as long as we identify the velocity field $\nabla\phi=-Q_u(f).$
\end{rem}
The concepts analogous to the Levi-Civita connection, geodesic concavity of functionals and entropy, and sectional curvature have appeared in optimal transport theory, in \cite{MR1842429}. Here, we reinterpret these findings through the point of views from $G_2$ geometry.

\subsection{Uniqueness and diffeomorphisms}	
\begin{prop}\label{uniqueness body}	
	The torsion-free $G_2$ structure in $\mathcal M_1$ is unique up to diffeomorphism.
		Specifically speaking, if we assume that $\vphi(t), t\in [0,1]$ is a geodesic \eqref{geodesic 1} in \lemref{equations geodesic} and $\vphi(0), \ \vphi(1)$ are both torsion-free,
		then the geodesic $\vphi(t)$ is generated by the orientation-preserving diffeomorphism.
	\end{prop}
\begin{proof}
			Let $\vphi(t)$ be a path of $G_2$ structures in $\mathcal M_1$ such that the end points are torsion-free and the volume functional is concave, i.e., the second-order derivatives of the volume functional along $\vphi(t)$ is non-positive. Then, we have $\Vol_{tt}=\Vol_{t}=0$.
			
		According to \propref{geodesic concavity body}, it suffices to solve the ordinary differential equation
		$K(u_t)=u_t-\frac{u'Q(u_t)}{u}=0,
		$ which is solved in \lemref{diffeomorphism PDE solution}.
		 Therefore, we conclude the proposition.
	\end{proof}

\begin{prop}\label{diffeomorphismsgeodesic}
Let $\sigma(t)$ be a family of orientation-preserving diffeomorphism. Then, the family $u(x,t)=\sigma^\ast u(x)$ 
satisfies the geodesic equation \eqref{geodesic 1} in \lemref{equations geodesic}.
\end{prop}
\begin{proof}
The canonical geodesic \eqref{geodesic 1} in \lemref{equations geodesic} is rewritten as		
\begin{align*}
		u_t \bigg(\log\frac{u_t}{u_x}\bigg)_t
			=\frac{-Ku_{tx}}{u_x}+\frac{u_t}{u} K+\bigg[K\frac{Q(u_t)}{u}\bigg]_x.
\end{align*}
By definition, we know $u(x,t)=H(x+C(t))$ with arbitrary positive smooth function $H$ and $\p_tC(t)>0$.
Then, $u(t)$ solves $K(u_t)=0$ and equivalently $u_t=u_x e^{C(t)}$. Thus, $u(t)$ solves the geodesic equation.		
\end{proof}
	
\subsection{Length contraction}	
Comparing the variational formula \eqref{eq:2nd_vol chi general} and the metric \eqref{eq:Dmetric}, we see that the Euler-Lagrange equation of the $\chi$ volume is
	$
	LQ_u(g)=\chi_{u}=0,
	$
	which is actually, after differentiating twice and using Definition \ref{operator L and Q},
	$g=\tilde \tau'=0$,  $\tilde \tau :=  \chi_{uu} u' u$.
	We notice that if we set
$\tilde u:= u\chi_{u}-\chi$, then $\tilde u'= \tilde \tau$. Therefore, the gradient flow $u_t=-g$, which will be called the $\chi$ flow, has the expression
	\begin{equation}\label{weighted gradient flow 1}
	\begin{split}
	u_t=-\tilde u''	= -  \tilde \tau' 
	=- (\chi_{uu}u)  u''  - (\chi_{uuu}  u + \chi_{uu} )(u')^2.
	\end{split}
	\end{equation}

When $\chi(u)=u^p$, the flow is reduced to $u_t=-(p-1)(u^p)'',$ which is called the 1-d fast diffusion equation $0<p<1$, or the porous medium equation $p> 1$.
When $\chi(u)=u(1-\log u)$, the $\chi$ flow coincides with the heat equation.

\begin{defn}
We define
\begin{itemize}
\item the Dirichlet norm of $u_t$: $\|u_t\|^2_{\mathcal G}:=\mathcal G_\vphi(u_t,u_t)= \int u^{-1} Q (u_t)^2$;
\item the energy $\mathcal E(u):= \int u^{-1} \tilde\tau^2 $,\quad $\mathcal F(\tilde\tau):=\int \eta(\tilde \tau)$;
\item the entropy $\mathcal F(u):= \int \eta(u)$,
\end{itemize}
where $\eta(y)$ is a weight function, which could be convex, e.g. $y\log y$ and $y^p, p\geq 1$, or concave.
\end{defn}
\begin{lem}$Q (u_t)=-\tilde\tau$
and $\|u_t\|^2_{\mathcal G}=\mathcal E(u)$, along the $\chi$ flow.
\end{lem}
	\begin{proof}
Plugging the flow equation \eqref{weighted gradient flow 1} into Definition \ref{operator L and Q}, we know $L(u_t)$ differs from $-\chi_{uu} u' u$ by a constant, but the constant does not affects the computation of $Q$. By Definition \ref{operator L and Q}, the normalisation constant $C_1[L(u_t)]$ required in $Q$ is $\int u^{-1}\chi_{uu} u' u=\int (\chi_{u})'=0$ as well. As a result, we obtain
$Q (u_t)= -\chi_{uu} u' u=-\tilde\tau$
and the expression of $\|u_t\|^2_{\mathcal G}$.
\end{proof}
	
	\begin{lem}\label{chi volume increase}
	Along the $\chi$ flow, 
	\begin{itemize}
	\item the $\chi$ volume functional is non-decreasing 
$\p_t\Vol_\chi=\mathcal E(u)\geq 0;$
	\item the energy $\mathcal F(\tilde\tau)$ is monotone, $\p_t \mathcal F(\tilde\tau)=\int \eta_{yy}\chi_{uu} (\tilde \tau')^2 u $;
	\item the entropy $\mathcal F(u)$ is monotone, $\p_t \mathcal F(u)=\int \eta_{yy}\chi_{uu}( u' )^2 u$.
	\end{itemize}
	\end{lem}
\begin{proof}
We use the integration by parts, the flow equation \eqref{weighted gradient flow 1} and insert $\tilde u_t=u_t\chi_{u}+u\chi_{uu}u_t-\chi_u u_t=u\chi_{uu}u_t=-u\chi_{uu}\tilde \tau'$,
\begin{align*}
&\p_t \mathcal F=\int \eta_y \tilde \tau_t=\int \eta_y \tilde u'_t
=\int \eta_{yy}  \tilde \tau' (u\chi_{uu}\tilde \tau')
\end{align*}
and $\p_t \mathcal F_{\eta}=\int \eta_y u_t =-\int  \eta_y (\chi_{uu} u' u)' =\int  \eta_{yy} \chi_{uu} (u')^2 u$.
\end{proof}

We apply \thmref{length contraction} to a family of $\chi$ flows. 
\begin{prop}\label{distance HS flow}
The distance is non-increasing along the $\chi$ flow,
\begin{align*}
\p_t \|u_s\|^2_{\mathcal G}
 =2\int_M \chi_{uu}  K(u_s)^2 \leq 0,\quad K(u_s):=u_s-\frac{Q(u_s) u_x}{u}.
\end{align*}
Furthermore, $\p_t^2\Vol_{\chi}=\p_t \mathcal E[u(t)]=2\int_M \chi_{uu}  K(u_t)^2 \leq 0$.
\end{prop}
\begin{proof}
Plugging the expression of the connection $D_t\vphi_s=-(u_{st}-U')d\theta^3\wedge dx^{03}$ obtained in
\propref{defn:con} to \thmref{length contraction}, we need to compute two terms, 
\begin{align*}
I=2\int L(u_{ts}) Q(u_s) u^{-1} ,\quad
II=\int  -u_t Q(u_s)^2 u^{-2} .
\end{align*}

We insert the linearised equation of the flow equation \eqref{weighted gradient flow 1}
\begin{align*}
u_{ts}=-(\chi_{uu} u' u)_s'=-(\chi_{uu} u_s u)'',\quad L(u_{ts})=-(\chi_{uu} u_s u)'
\end{align*}
to the first term $I$ and apply the integration by parts,
\begin{align*}	
	=2\int \chi_{uu} u_s u [-u^{-2}u' Q(u_s)+u^{-1}u_s] 
		=2\int \chi_{uu}  [-u^{-1}u' Q(u_s) u_s+u^2_s] .
\end{align*}

We insert the flow \eqref{weighted gradient flow 1} and apply the integration by parts to $II$, 
\begin{align*}
II= \int   \tilde\tau' Q(u_s)^2 u^{-2}= 2\int  \chi_{uu}  [ u^{-2} (u')^2 Q(u_s)^2
-u^{-1} u' Q(u_s)u_s] .
\end{align*}

In conclusion, we have
$\p_t  \|u_s\|^2_{\mathcal G}
=2\int \chi_{uu}  [u_s-u^{-1}u' Q(u_s) ] ^2
\leq 0,$
which implies the required formula. 
The monotonicity of the energy $\|u_t\|^2_{\mathcal G}$ of the flow is obtained, if we choose  $u_0(t)=u(t)$ and $u_1(t)=u(t+1)$ in the argument above.
\end{proof}

\subsection{Curvature of $(\mathcal M_1,\mathcal G)$}
We set
$
X:=-\underline{f}=-u_t$, $Y:=-\underline{g}=-u_s.
$
\begin{prop}\label{Curvature body}
The sectional curvature
\begin{align*}
K(X,Y)&=\frac{3}{4}\bigg[\int \frac{fQ(g)-Q(f)g}{u^2}\bigg]^2 \big[ \int u^{-1} \big]^{-1}\geq 0.
\end{align*}
\end{prop}

To reduce the complexity of the computation, we use the lower index $u$, which indicates the quantities divided by $u$, i.e. $f_u:= \frac{f}{u}$, we also introduce the following notations:
$V_u:=[\int u^{-1}]^{-1}$, $b:= \frac{u'}{u}$, $c_f:= \frac{Q(f)}{u}$.

We will compute 
\begin{align}\label{Sectional curvature K}
K(X,Y)&=-[\mathcal G(D_YD_XY,X)-\mathcal G(D_XD_YY,X)].
\end{align}
\subsubsection{Compute $D_YD_XY$}
Because $L$ Definition \ref{operator L and Q} and $Q$ Definition \ref{operator L and Q} are linear, we compute from Definition \ref{defn:con} that $U=fc_{g}
+c_{f}g- c_{f}c_{g} u'$ and
\begin{equation}\label{LQDXY}
\begin{split}
&D_XY=Y_t+U'=-g_t+U' , \quad 
L(D_XY)
=-L(g_t)
+U
-\int U,\\
&Q(D_XY)
=-Q(g_t)
+U
-\int u^{-1}U V_u
=\\
&-Q(g_t)
+fc_{g}
+c_{f}g
- c_{f}c_{g} u'
+(\int -f_uc_{g}-c_{f} g_u+ c_{f}c_{g} b ) V_u.
\end{split}
\end{equation}		
We replace $f, g$ with $g, -D_XY$ respectively in $U$ and let
\begin{align*}
W:= g_uQ(-D_XY)+(-D_XY)c_{g}- c_{g}Q(-D_XY) b.
\end{align*} Then by definition $D_YD_XY=(D_XY)_s+W'=Y_{ts}+U_s'+W'$,
\begin{align*}
L(D_YD_XY)&=L(Y_{ts})
+U_s+W-\int(U_s+W),\\
Q(D_YD_XY)
&=Q(Y_{ts})
+U_s+W
-[\int u^{-1} (U_s+W) ] V_u.		
\end{align*}	
Putting together, 		
$
			\mathcal G(D_YD_XY,X)
			=\int Q(-D_YD_XY)c_{f}
$
		\begin{equation}\label{G(DYDXY,X)}
		\begin{split}
=-\int  Q(Y_{ts}) c_{f}
-\int (U_s+W) c_{f}  .
\end{split}
		\end{equation}
\begin{lem}
\begin{align*}
&(U_s+W)u^{-1}
=2(f_s)_uc_{g}
+f_uc_{g_s}
+c_{f} (g_s)_u
+2c_{f_s}g_u\\
&
-4   f_uc_{g}g_u
-2c_{f}g_u^2
- 2 c_{f_s} c_{g} b
- c_{f} c_{g_s}b
-2 c_{f}    c_{g} (g')_u
- (f')_u c_{g}^2\\
&+6 c_{f} g_u c_{g} b
+3 f_uc_{g}^2b
+ c_{f} c_{g}^2 (u'')_u
-3 c_{f} c_{g}^2 b^2\\
&+\big[f_u( \int g_u c_{g})
+c_{f} b ( \int -g_u c_{g})
+g_u(\int f_uc_{g}
+2c_{f}g_u
- c_{f}c_{g} b)\\
&+c_{g} b(\int -f_uc_{g}-2c_{f}g_u+ c_{f}c_{g} b) \big]V_u u^{-1}.
\end{align*}
\end{lem}
\begin{proof}
We compute 
\begin{align*}
&U'=(\frac{f'Q(g)}{u}+\frac{fg}{u}-\frac{fQ(g)u'}{u^2})
+(\frac{g'Q(f)}{u}
+\frac{gf}{u}
-\frac{gQ(f)u'}{u^2})\\
&- \frac{fQ(g) u'}{u^2}-\frac{Q(f)g u'}{u^2}-\frac{Q(f)Q(g) u''}{u^2}
+2 \frac{Q(f)Q(g) (u' )^2}{u^3}\\
&=[(f')_uc_{g}+2f_ug_u+(g')_uc_{f}
-2 f_u c_{g} b-2c_{f} g_u b
-c_{f} c_{g} (u'')_u
+2 c_{f} c_{g} b^2 ] u
\end{align*}
and
\begin{align*}
U_s&=(\frac{f_sQ(g)}{u}+\frac{fQ(g)_s}{u}-\frac{fQ(g)u_s}{u^2})+(\frac{g_sQ(f)}{u}+\frac{gQ(f)_s}{u}
-\frac{gQ(f)u_s}{u^2})\\
&- \frac{Q(f)_sQ(g) u'}{u^2}
-\frac{Q(f)Q(g)_s u'}{u^2}
-\frac{Q(f)Q(g) u'_s}{u^2}
+2 \frac{Q(f)Q(g) u' u_s}{u^3}\\
&=[(f_s)_uc_{g}
+f_uc_{g_s}
+c_{f}(g_s)_u
+c_{f_s} g_u
-f_uc_{g}g_u
-c_{f}g_u^2 \\
&- c_{f_s}c_{g} b
-c_{f}c_{g_s} b
-c_{f}c_{g} (g')_u
+2 c_{f}c_{g}b g_u] u\\
&+f_u\int g_u c_{g} V_u
+g_u\int g_u c_{f}V_u
- c_{g} b \int  g_u c_{f}V_u
-c_{f} b \int g_u c_{g}V_u.
\end{align*}
In the last equality, we use
\begin{equation}\label{Qst}
\begin{split}	
&Q(g)_t=Q(g_t)+\int  f_u c_{g} V_u,\quad Q(g)_s=Q(g_s)+\int  g_u c_{g} V_u,\\
&Q(f)_t=Q(f_t)+\int  f_u c_{f} V_u,\quad Q(f)_s=Q(f_s)+\int  g_u c_{f} V_u.
\end{split}
\end{equation}

We compute each term in $W$, using \eqref{LQDXY} and $g_t=f_s$. 
The term $-g_uQ(D_XY)$ is equal to
\begin{align*}
&=u(g_uc_{f_s}
-f_uc_{g}g_u
-c_{f}g_u^2
+ c_{f}c_{g}g_u b)
+g_u(\int f_uc_{g}
+c_{f}g_u
- c_{f}c_{g} b) V_u.
\end{align*}
While, $(-D_XY)c_{g}
=-c_{g}(-g_t+U')$ is
\begin{align*}
&=u[(f_s)_u c_{g}  
-(f')_uc_{g}^2-2f_ug_uc_{g}-(g')_uc_{f}c_{g}\\
&+2 f_u c_{g}^2b+2c_{f}g_uc_{g}b
+ c_{f}c_{g}^2 (u'')_u
-2 c_{f}c_{g}^2 b^2 ]
\end{align*}
and $c_{g} b Q(D_XY)$ becomes
\begin{align*}
&=u(-c_{f_s} c_{g} b
+f_uc_{g}^2b+c_{f}g_uc_{g}b- c_{f}c_{g}^2 b^2)
+c_{g} b(\int -f_uc_{g}-c_{f}g_u+ c_{f}c_{g} b) V_u.
\end{align*} 
The confusion follows by collecting the like terms together.
\end{proof}

\subsubsection{Compute $D_XD_YY$}
We set
$
 \tilde U:=2\frac{gQ(g)}{u}- \frac{Q^2(g) u'}{u^2}.
$
 Then, we apply the formula $Q(\cdot)=L(\cdot) 
-[\int u^{-1} L(\cdot) ]  [\int u^{-1}]^{-1}$,
\begin{align*}
D_YY&=Y_s+\tilde U',\quad
L(D_YY)=L(Y_s)+\tilde U-\int \tilde U,\\
Q(D_YY)
&=Q(Y_s)
+u(2g_uc_{g}- c_{g}^2 b)
+(\int -2g_uc_{g}+ c_{g}^2 b ) V_u.
\end{align*}			
We set the auxiliary function
\begin{align*}
\tilde W:= f_uQ(-D_YY)+(-D_YY) c_{f}- c_{f}Q(-D_YY) b.
\end{align*}
Then $D_XD_YY=(D_YY)_t+P(X,D_YY)
			=Y_{st}+\tilde U_t' +\tilde W' $,
\begin{align*}
L(D_XD_YY)&=L(Y_{st})
+\tilde U_t+\tilde W-\int(\tilde U_t+\tilde W),\\
Q(D_XD_YY)&=Q(Y_{st})	+\tilde U_t+\tilde W
	 -[\int u^{-1}(\tilde U_t+\tilde W) ] V_u.
\end{align*}
As a result, $\mathcal G(D_XD_YY,X)
=\int Q(-D_XD_YY)c_{f}$ is
\begin{equation}\label{G(DXDYYX)}
\begin{split}
=-\int Q(Y_{st}) c_{f}
-\int (\tilde U_t+\tilde W) c_{f}.
\end{split}
\end{equation}

\begin{lem}
\begin{align*}
&(\tilde U_t+\tilde W)u^{-1}
=2(f_s)_uc_{g}
+2c_{f_s}g_u
+f_u c_{g_s}
+  c_{f} (g_s)_u 
\\
&-4f_ug_uc_{g}
-2 c_{f_s}c_{g} b
-(f')_u c_{g}^2 
  -2c_{f} (g')_u c_{g}
   -2c_{f} g_u^2 
   -c_{f} c_{g_s} b
\\
&+3f_u c_{g}^2 b 
+6c_{f} g_uc_{g}b
+c_{f} c_{g}^2 (u'')_u 
-3c_{f} c_{g}^2b^2 \\
&+[f_u(\int 2g_uc_{g}- c_{g}^2 b )
+c_{f} b (\int -2g_uc_{g}+ c_{g}^2 b ) \\
&+g_u (\int 2f_u c_{g})  
+ c_{g} b (\int -2f_u c_{g}) ]  V_u u^{-1}  .
\end{align*}
\end{lem}
\begin{proof}
We compute directly, with the help of \eqref{Qst} and $g_t=f_s$,
\begin{align*}
\tilde U'
&=u[2(g')_u c_{g}+2g_u^2-4g_uc_{g}b
-c_{g}^2 (u'')_u+2c_{g}^2b^2],\\
\tilde U_t
&=(2\frac{g_tQ(g)}{u} +2\frac{gQ(g)_t}{u}-2\frac{gQ(g)f}{u^2})
-2 \frac{Q(g)Q(g)_t u'}{u^2}-\frac{Q^2(g) f'}{u^2}+2\frac{Q^2(g) u' f}{u^3}\\
&=u[2(f_s)_uc_{g}
+2c_{f_s} g_u
-2f_ug_uc_{g}
-2 c_{f_s}c_{g} b
-(f')_u c_{g}^2
+2f_u c_{g}^2 b] \\
&+2g_u \int f _uc_{g} V_u 
-2 c_{g} b \int f_u c_{g} V_u .
\end{align*}
We compute $\tilde W$. The first one $-f_uQ(D_YY)$ equals to
\begin{align*}
&=u( f_u c_{g_s}   -  2f_ug_uc_{g} + f_uc_{g}^2 b)
+f_u (\int 2g_uc_{g}- c_{g}^2 b )V_u.
\end{align*}
The second one $- c_{f} D_YY=- c_{f}[Y_s+\tilde U']$ is
\begin{align*}
&=  u [c_{f} (g_s)_u -2c_{f}(g')_u c_{g}  -2c_{f}g_u^2 
+4c_{f}g_uc_{g}b
+c_{f}c_{g}^2 (u'' )_u
-2c_{f}c_{g}^2b^2 ].
\end{align*}
The third one $c_{f}b Q(D_YY) $ becomes
\begin{align*}
&=u[-c_{f} c_{g_s} b 
+2c_{f}g_uc_{g} b
- c_{f}c_{g}^2b^2 ]
+c_{f} b (\int -2g_uc_{g}+ c_{g}^2 b)V_u.
\end{align*}
Adding up these terms, we have the required formula.
\end{proof}

\subsubsection{Compute $K(X,Y)$: proof of \thmref{Curvature body}}
We denote $$Z:=(U_s+W)-(\tilde U_t+\tilde W).$$ 
We combine the computation in the previous two subsections.
\begin{align*}
\frac{Z}{V_u}
&=f_u [\int -g_uc_{g}+ c_{g}^2 b ] 
+c_{f} b [\int g_uc_{g}-  c_{g}^2 b ] \\
&+ g_u [\int -f_uc_{g}
+2 c_{f}g_u 
-  c_{f}c_{g} b ] 
+ c_{g} b [\int  f_uc_{g} -2 c_{f}g_u + c_{f}c_{g} b ] .
\end{align*}

We denote 
$\b_1:=\int f_uc_{g},\quad \b_2:=\int c_{f}g_u,\quad \b_3:=\int g_uc_{g}.$ Then by integration by parts
\begin{align*}
&\b_4:=\int  c_{f}c_{g} b  =3\b_4-(\b_1+\b_2)\text{ gives } \b_4=\frac{\b_1+\b_2}{2},\\
&\b_5:=\int c_{g}^2 b =3\b_5-2\b_3 \text{ implies } \b_5=\b_3.
\end{align*}
Thus, we arrive at
$
\frac{Z}{V_u}
= (\b_1-2\b_2+\b_4)[ c_{g} b - g_u ] .
$
We compute $\b_1-2\b_2+\b_4=\frac{3(\b_1-\b_2)}{2}$ and
\begin{align*}
\int  \frac{Z}{V_u} c_{f}
&=\frac{3(\b_1-\b_2)}{2}(\b_4-\b_2) =\frac{3(\b_1-\b_2)^2}{4} .
\end{align*}
Therefore, the sectional curvature $K(X,Y)$ is obtained from putting \eqref{G(DYDXY,X)}\eqref{G(DXDYYX)} into \eqref{Sectional curvature K},
$
K(X,Y)
=\int Z c_{f}.
$

	
\section{Two constructions}\label{Two constructions}
Now we are given a hyper-symplectic triple
	\begin{align}\label{hypers simple}
		\underline\om:=(\om_1,\om_2,\om_3);\quad \om_i:=\om_i^K+(u_1-1)dx^{0i}.
	\end{align}

	
\subsection{Concavity and formal Riemannian structure}\label{concavity}
The closed $G_2$ structure \eqref{vphi} determined by $\underline\om_i$ in \eqref{HS vector} is 
	\begin{align}\label{vphi simple}
		\vphi_i =\vphi^K-(u_i-1)d\theta^i\wedge dx^{0i},
	\end{align}
which forms a triple $\underline\vphi:=(\vphi_1,\vphi_2,\vphi_3)$. Note that in \eqref{vphi simple}, the index $i$ is not taken from the Einstein summation.
\begin{defn}\label{defn:space 3}
Let $\mathcal{M}_3$ be the space of the closed $G_2$ structures $\underline\vphi:=(\vphi_1,\vphi_2,\vphi_3)$, each component of the form \eqref{vphi simple}, in the class $[\underline{\vphi}^K]$ on $\mathbb{T}^7$.
\end{defn}
We further construct the Riemannian structure for $\mathcal{M}_3$. The idea is to project the triple $\underline\om$ to each component $i$ and use the results from the previous sections.

We adapt the constructions from Section \ref{HS and G2} to the hyper-symplectic structure \eqref{hypers simple}.
We write 
\begin{equation}\label{vAf simple}
\begin{split}
&\omega_i=dx^0\wedge dx^{i-1}+u_idx^0\wedge dx^i+dx^0\wedge dx^{i+1}-dx^{i-1}\wedge dx^{i+1}.
\end{split}
\end{equation}
\begin{itemize}
\item
The metrics $g^i_{3}= u_i^{\frac{2}{3}} (d\theta^i)^2+u_i^{-\frac{1}{3}}  [(d\theta^{i-1})^2+(d\theta^{i+1})^2]$,
\begin{align*}
	& g^i_{4}=u_i^{\frac{2}{3}}[ (dx^0)^2+ (dx^i)^2]+u_i^{-\frac{1}{3}}  [(dx^{i-1})^2+ (dx^{i+1})^2] .
\end{align*} 
The volume element 
$\vol_{\vphi_i}=u_i^{\frac{1}{3}} dx^{0123}\wedge d\theta^{123}.$
\item The torsion form $\tau=\sum_i \tau_id\theta^i\wedge d x^i$ \eqref{eq:tau} satisfies that
$\tau_1=\tau_2=-\frac{1}{3}u_i^{-\frac{5}{3}}u_i'$,  $\tau_3=\frac{2}{3}u_i^{-\frac{2}{3}}u_i'$.
\item
The tangent vector is also a triple $-\underline{f}$, each component is 
\begin{align*}
- f_i d\theta^i\wedge dx^{0i} , \quad  f_i\in C_0^\infty(S^1).
\end{align*}
\item 
The $Q$ operator in Definition \ref{operator L and Q} now becomes
$
Q_i(g_i):= L(g_i)+C_1(g_i)$, where $C_1(g_i):= -[\int_{S^1}A_iL(g_i) ][\int_{S^1}A_i]^{-1}
$ and $A_i:= u_i^{-1}$.

\item The Dirichlet metric \eqref{metric} (see \eqref{eq:Dmetric}) is
\begin{equation}\label{eq:Dmetric simple}
\begin{split}
\underline{\mathcal G}_{\underline\vphi}(Y,Z)
&:= \sum_i \mathcal G_{\vphi_i},\quad \mathcal G_{\vphi_i}:= \int_{S^1} A_i Q_i(g_i)Q_i(h_i).
\end{split}
\end{equation}

\item The Levi-Civita connection (see Definition \ref{defn:con}) is $D_XY:= Y_t+\sum_i U_i ' d\theta^i\wedge dx^{0i}$, where 
\begin{equation}\label{LP simple}
\begin{split}
U_i&:= 
\frac{f_iQ(g_i)+g_iQ(f_i)}{2u_i}+ \frac{u_i}{2}\bigg[\frac{Q_i(f_i)Q_i(g_i)}{u_i^2}\bigg]'\\&=\frac{f_i Q_i(g_i)+g_i Q_i(f_i)}{u_i}-\frac{Q_i(f_i) Q_i(g_i) u_i'}{u_i^2} .
\end{split}
\end{equation}

\item 
The geodesic equation (see \eqref{geodesic 1} in \lemref{equations geodesic}) is automatically
\begin{equation}\label{Lp abc simple}
\begin{split}
(f_i)_t= \tilde U_i',\quad \tilde U_i&=
(2a_ic_i
-b_ic_i^2) u_i,
\end{split}
\end{equation}			
with the notions
	$a_i:= f_iu_i^{-1}$, $b_i:= u_i'u_i^{-1}$, $c_i:= Q_i(f_i)u_i^{-1}$.
\item The gradient flow for the $\underline\chi$ functional \eqref{3 weighted volume functional} is
	\begin{align}\label{weighted gradient flow 3}
	(u_i)_t=-(\chi_i'' u_i' u_i)',\quad i=1,2,3.
	\end{align}
\end{itemize}	
	

Following the same routine of the proofs in the previous sections, we achieve the same Riemannian structure for $(\mathcal{M}_3,\underline{\mathcal G})$.
	
\begin{thm}\label{M3 structure}
	The space $(\mathcal M_3,\underline{\mathcal G})$ has formal Riemannian structure:
\begin{enumerate}
	\item
the connection \eqref{LP simple} is the Levi-Civita connection for the Dirichlet metric $\underline{\mathcal G}$ \eqref{eq:Dmetric simple} defined in $\mathcal{M}_3$;
		\item
	The $\underline\chi$ volume functional \eqref{3 weighted volume functional} is geodesically concave.
	In fact, if we let $\underline\vphi(t)$ be the geodesic \eqref{Lp abc simple}, then\begin{align*}
\p_t^2\Vol_{\underline \chi}&=\sum_i\int \chi_i''  K_i^2, \quad K_i[(u_i)_t]:=u_i(a_i-b_ic_i );
\end{align*}
\item the $\underline\chi$ flow  \eqref{weighted gradient flow 3} decreases the distance;
\item the sectional curvature is nonnegative.
\end{enumerate}
\end{thm}
\begin{cor}\label{geodesically concave 3}
Let $\chi_i(\cdot)=(\cdot)^{\frac{1}{3}}$. Then, the volume functional $\Vol_{\frac{1}{3}}=\sum_i\int u_i^{\frac{1}{3}}=\sum_i\Vol(\vphi_i)$ is geodesically concave in $(\mathcal M_3,\underline{\mathcal G})$.
\end{cor}

\subsection{A counterexample: non-concavity}\label{non-concavity}
We present a counterexample that violates the geodesic concavity of the volume functional along the canonical geodesic. We follow the construction in Section \ref{HS and G2}.

Given the hyper-symplectic triple $\underline{\om}$, there is another associated closed $G_2$ structure  
	\begin{align}\label{vphi simple simple}
		\vphi=d\theta^{123}-\sum_i d\theta^i\wedge \omega_i =\vphi^K-\sum_i (u_i-1)d\theta^i\wedge dx^{0i}.
	\end{align}
\begin{defn}\label{defn:space 3 tilde}
	Let $\widetilde{\mathcal{M}}_3$ be the space of the closed $G_2$ structures on $\mathbb{T}^7$, which is of the form \eqref{vphi simple simple} in the class $[\vphi^K]$.
\end{defn}
\begin{itemize}
\item The tangent space $T\widetilde{\mathcal{M}}_3=\big\{-\underline{f}\mid  \underline{f}:=\sum_i f_i d\theta^i\wedge dx^{0i}, \  f_i\in C_0^\infty(S^1)\big\}$.
		
\item The volume element 
$\vol_\vphi=w \cdot \vol_E$, $w:=(u_1u_2u_3)^{\frac{1}{3}}$.
We write $v_i:= u_i w^{-1}$, the associated metrics 
$g_{4}=diag\{w^2,v_1,v_2,v_3\}$ and
$g_{3}=diag\{v_1,v_2,v_3\}$.
\item The torsion form $\tau=\sum_i \tau_id\theta^i\wedge dx^i$, $
\tau_i=v_i'w^{-1}.$

\item	We write $f_i=(wv_i)_t$ and use the notations
			\begin{align}\label{A and B simple}
			A_i:= wv_i^{-2},\quad B_i(\underline{f}):= -\sum_j \epsilon_{ij} f_j(v_i^2v_j)^{-1},\quad i,j=1,2,3.
			\end{align} 
			Here $\epsilon_{ij}=1, i=j$; and $\epsilon_{ij}=-1, i\neq j$.
		It holds $(A_i)_t=B_i(\underline{f})$,
		\begin{align*}
		w_t=3^{-1}\sum_i f_iv_i^{-1},\quad (v_i)_t =-v_i(3w)^{-1}\sum_jf_jv_j^{-1}+f_iw^{-1}.
		\end{align*}

\item The Hodge Laplacian operator has the following expression
$\Delta \underline{h}=-d\big\{\sum_i \big[\big(h_i(wv_i^2)^{-1}\big)' A_i^{-1}\big]dx^i\wedge d\theta^i \big\}$.
\item The $Q$ operator is
$Q_i(g_i):= L(g_i)+C_1(g_i)$, where the constant $C_1(g_i):= -[\int_{S^1}A_iL(g_i) ][\int_{S^1}A_i]^{-1}$.
\item Let $C_2:= -\big\{\int_{S^1}wv^2_i\cdot L\big[A_i Q_i(g_i)\big] \big\}[\int_{S^1}wv^2_i]^{-1}$, the Green operator is
$G\underline{g}=-\sum_i wv^2_i  \big\{L\big[A_i Q_i(g_i)\big]+C_2\big\}d\theta^i\wedge dx^{0i} $.
\item
The Dirichlet metric \eqref{metric} becomes
$$\mathcal G_\vphi(Y,Z)
=\sum_i \int_{S^1} A_iQ_i(g_i)Q_i(h_i).$$
\item
The Levi-Civita connection is $D_XY:= Y_t+ U_i '  d\theta^i\wedge dx^{0i}$, where 
\begin{equation*}
\begin{split}
U_i:=\frac{-1}{2A_i}& \bigg[B_i(\underline{f}) Q_i(g_i)+B_i(\underline{g}) Q_i(f_i)
			- \sum_j\epsilon_{ij}\left( \frac{Q_j(f_j)Q_j(g_j) u}{u_iu_j^2}\right)' \bigg].
\end{split}
\end{equation*}
\end{itemize}

\begin{lem}The geodesic equation satisfies that
\begin{equation}\label{Lp abc simple simple}
\begin{split}
(f_i)_t= \tilde U_i',\quad \tilde U_i=\sum_j\epsilon_{ij} 
\big[a_jc_i+a_jc_j+\frac{c_j^2 }{2}\sum_{k\neq i} b_k
-c_j^2b_j\big] u_i 
\end{split}
\end{equation}		
with the notions
$	a_i:= \frac{f_i}{u_i}$, $b_i:= \frac{u_i'}{u_i}$, $c_i:= \frac{Q_i(f_i)}{u_i}$, $b_{ij}:= b_i \cdot c_j.$
\end{lem}
\begin{proof}
We further compute $\frac{u'}{u}=\sum_k\frac{u_k'}{u_k}$ and 
\begin{align*}
&U_i=-\frac{1}{2}\frac{u_i^2}{u}
			\bigg[ \sum_j -\epsilon_{ij}\frac{f_j u}{u_i^2u_j}\cdot Q_i(g_i)
			+\sum_j -\epsilon_{ij}\frac{g_j u}{u_i^2u_j}\cdot Q_i(f_i)\\
			&+\sum_j-\epsilon_{ij}
			 \bigg( \frac{f_j\cdot Q_j(g_j) u}{u_iu_j^2}
			+\frac{Q_j(f_j)\cdot g_j u}{u_iu_j^2}
			+\frac{Q_j(f_j)\cdot Q_j(g_j) u'}{u_iu_j^2}\\
			&
			 -\frac{Q_j(f_j)\cdot Q_j(g_j) u u_i'}{u_i^2u_j^2}
			 -2\frac{Q_j(f_j)\cdot Q_j(g_j) u u_j'}{u_iu_j^3}\bigg)\bigg]\\
			 &=\frac{1}{2}\sum_j \epsilon_{ij}
			\bigg[\frac{f_j}{u_j}\cdot Q_i(g_i)
			+\frac{g_j}{u_j}\cdot Q_i(f_i)
			+ \frac{f_j\cdot Q_j(g_j) u_i}{u_j^2}
			+\frac{Q_j(f_j)\cdot g_j u_i}{u_j^2}\\
			&+\frac{Q_j(f_j)\cdot Q_j(g_j) u_i u'}{u u_j^2}
			 -\frac{Q_j(f_j)\cdot Q_j(g_j) u_i'}{u_j^2}
			 -2\frac{Q_j(f_j)\cdot Q_j(g_j)u_i u_j'  }{u_j^3}\bigg].
			 \end{align*}	
\end{proof}	
	

\begin{lem}\label{weighted volume}Let $\vphi(t)$ be the geodesic \eqref{Lp abc simple simple}, $r=\chi_{uu}  $, $s=\chi_{uu}  +  \chi_{u} u^{-1}$. Then along $\vphi(t)$, we have
\begin{align*}
\p_t^2\Vol_\chi&=\int u^2 \big\{r\sum_i a_i^2+2s \sum_{i<j}a_ia_j
-\sum_i [r b_i + s \sum_{l\neq i} b_l  ] \frac{\tilde U_i}{u_i}\big\} .
\end{align*}
\end{lem}
\begin{proof}
According to \lemref{eq:2nd_vol chi general},  the second variation of the $\chi$ volume functional is
$
\p_t^2\Vol_\chi=\int \chi_{uu} u_t^2+\chi_{u} u_{tt}
$
and by the notions \eqref{Lp abc simple simple},
\begin{align*}
u_t=u\sum_i a_i,\quad 
u_{tt}=u[(\sum_i a_i)^2-\sum_i a_i^2+\sum_i\frac{(u_i)_{tt}}{u_i}].
\end{align*}
Then, integrate by parts and plug the geodesic $(u_i)_{tt}=\tilde U_i'$ \eqref{Lp abc simple simple},
\begin{align*}
&\int \chi_{u} u \sum_i\frac{(u_i)_{tt}}{u_i}
=\int\sum_i\frac{ \chi_{u} u }{u_i}  \tilde U_i'
=-\int\sum_i (\frac{ \chi_{u} u }{u_i})'  \tilde U_i.
\end{align*}
We further compute that $u'=u\sum_k b_k$,
\begin{align*}
(\frac{ \chi_{u} u }{u_i})' u_i 
&=[\frac{ \chi_{uu} u' u }{u_i}+ \frac{ \chi_{u} u' }{u_i}-\frac{ \chi_{u} u u_i'}{u^2_i}]u_i
= \chi_{uu} u' u +  \chi_{u} u' - \chi_{u} u b_i\\
&= (\chi_{uu} u^2 +  \chi_{u} u)\sum_l b_l - (\chi_{u} u) b_i
= (r u^2 ) b_i+ (s u^2 )\sum_{l\neq i} b_l .
\end{align*}
Add them together, we conclude the lemma.
\end{proof}

\begin{thm}\label{general weight}
Along the geodesic \eqref{Lp abc simple simple}, $\p_t^2\Vol_\chi=\int u^2 Q$, where
\begin{align*}
Q&=
r[a_1^2+a_2^2+a_3^2]+2s[ a_1a_2+a_2a_3+a_1a_3]
+(r-s)[b_{11}^2 +b_{22}^2 +b_{33}^2]\\
&+[-2(r-s)a_1b_{11}+sa_1b_{12}+sa_1b_{13}]
+[(r-s)a_1b_{21}+ra_1b_{22}+sa_1b_{23}]\\
&+[(r-s)a_1b_{31}+sa_1b_{32}+ra_1b_{33}]
+[ra_2b_{11}+(r-s)a_2b_{12}+sa_2b_{13}]\\
&+[sa_2b_{21}-2(r-s)a_2b_{22}+sa_2b_{23}]
+[sa_2b_{31}+(r-s)a_2b_{32}+ra_2b_{33}]\\
&+[ra_3b_{11}  +sa_3b_{12}   +(r-s)a_3b_{13}]
+[sa_3b_{21}+ra_3b_{22}+(r-s)a_3b_{23}]\\
&+[sa_3b_{31}+sa_3b_{32}-2(r-s)a_3b_{33}]
+(r-s)[-b_{11}b_{21}
-b_{12}b_{22}
+b_{13}b_{23}]\\
&+(r-s)[-b_{11}b_{31}
+b_{12}b_{32}
-b_{13}b_{33} ]
+(r-s)[b_{21}b_{31}
-b_{22}b_{32}
-b_{23}b_{33 }].
\end{align*}
The matrix $Q$ is not negative-semidefinite.

 In particular, all $\chi$ volume functional, including the volume functional i.e. $\Vol(\vphi)=\int (u_1u_2u_3)^{\frac{1}{3}},$ is not concave along the canonical geodesic \eqref{Lp abc simple simple} in the space $(\widetilde{\mathcal{M}}_3,\mathcal G)$.
\end{thm}
\begin{proof}
The formula of $Q$ is obtained from expanding \lemref{weighted volume}. We have added the corresponding matrix $A$ in Section \ref{Matrix}. The trace of $A$ is equal to $12r$. From the left top corner of $A$, we can see that the matrix $A$ has eigenvalues including
$\frac{s-r}{2}$, $\frac{r-s}{2}.$ That means if $Q$ is negative semi-definite, $s=r\leq 0$. It infers that the weight $\chi$ must have a vanishing first derivative and a semi-negative second derivative, which is impossible for a positive concave function.
\end{proof}

\begin{rem}
Compared with \corref{geodesically concave 3}, the canonical geodesic we constructed are different.
\end{rem}


\subsubsection{Partitioned Matrices}\label{Matrix}
We partition the matrix \(A\) into four sub-matrices:
\(A=\begin{bmatrix}A_{11}&A_{12}\\A_{21}&A_{22}\end{bmatrix}\), where
$
A_{11} = 
\begin{bmatrix}
r & s & s \\
s & r & s \\
s & s & r
\end{bmatrix}
$, $A_{21}=A^T_{12}$,
\[
A_{12} = 
\begin{bmatrix}
-(r - s) & \frac{s}{2} & \frac{s}{2} & \frac{r - s}{2} & r/2 & \frac{s}{2} & \frac{r - s}{2} & \frac{s}{2} & \frac{r}{2} \\
\frac{r}{2} & \frac{r - s}{2} & \frac{s}{2} & \frac{s}{2} & -(r - s) & \frac{s}{2} & \frac{s}{2} & \frac{r - s}{2} & \frac{r}{2} \\
\frac{r}{2} & \frac{s}{2} & \frac{r - s}{2} & \frac{s}{2} & \frac{r}{2} & \frac{r - s}{2} & \frac{s}{2} & \frac{s}{2} & -(r - s)
\end{bmatrix}
\]

\[
A_{22} = 
\begin{bmatrix}
0 & 0 & 0 & -\frac{r - s}{2} & 0 & 0 & -\frac{r - s}{2} & 0 & 0 \\
0 & 0 & 0 & 0 & -\frac{r - s}{2} & 0 & 0 & -\frac{r - s}{2} & 0 \\
0 & 0 & 0 & 0 & 0 & \frac{r - s}{2} & 0 & 0 & -\frac{r - s}{2} \\
-\frac{r - s}{2} & 0 & 0 & 0 & 0 & 0 & \frac{r - s}{2} & 0 & 0 \\
0 & -\frac{r - s}{2} & 0 & 0 & 0 & 0 & 0 & -\frac{r - s}{2} & 0 \\
0 & 0 & \frac{r - s}{2} & 0 & 0 & 0 & 0 & 0 & -\frac{r - s}{2} \\
\frac{r - s}{2} & 0 & 0 & \frac{r - s}{2} & 0 & 0 & 0 & 0 & 0 \\
0 & -\frac{r - s}{2} & 0 & 0 & -\frac{r - s}{2} & 0 & 0 & 0 & 0 \\
0 & 0 & -\frac{r - s}{2} & 0 & 0 & -\frac{r - s}{2} & 0 & 0 & 0
\end{bmatrix}.
\]



\begin{bibdiv}
		\begin{biblist}
			\bib{MR916718}{article}{
				author={Bryant, Robert L.},
				title={Metrics with exceptional holonomy},
				journal={Ann. of Math. (2)},
				volume={126},
				date={1987},
				number={3},
				pages={525--576},
				issn={0003-486X},
				review={\MR{916718}},
				doi={10.2307/1971360},
			}
			
			\bib{MR2282011}{article}{
				author={Bryant, Robert L.},
				title={Some remarks on $G_2$-structures},
				conference={
					title={Proceedings of G\"{o}kova Geometry-Topology Conference 2005},
				},
				book={
					publisher={G\"{o}kova Geometry/Topology Conference (GGT), G\"{o}kova},
				},
				date={2006},
				pages={75--109},
				review={\MR{2282011}},
			}
%

\bib{MR3412344}{article}{
   author={Calamai, Simone},
   author={Zheng, Kai},
   title={The Dirichlet and the weighted metrics for the space of K\"{a}hler
   metrics},
   journal={Math. Ann.},
   volume={363},
   date={2015},
   number={3-4},
   pages={817--856},
   issn={0025-5831},
   review={\MR{3412344}},
   doi={10.1007/s00208-015-1188-x},
}

			\bib{MR1701920}{article}{
   author={Donaldson, Simon},
   title={Moment maps and diffeomorphisms},
   note={Sir Michael Atiyah: a great mathematician of the twentieth
   century},
   journal={Asian J. Math.},
   volume={3},
   date={1999},
   number={1},
   pages={1--15},
   issn={1093-6106},
   review={\MR{1701920}},
   doi={10.4310/AJM.1999.v3.n1.a1},
}
			\bib{MR2313334}{article}{
   author={Donaldson, Simon},
   title={Two-forms on four-manifolds and elliptic equations},
   conference={
      title={Inspired by S. S. Chern},
   },
   book={
      series={Nankai Tracts Math.},
      volume={11},
      publisher={World Sci. Publ., Hackensack, NJ},
   },
   date={2006},
   pages={153--172},
   review={\MR{2313334}},
}

			\bib{MR3702382}{article}{
				author={Donaldson, Simon},
				title={Adiabatic limits of co-associative Kovalev-Lefschetz fibrations},
				conference={
					title={Algebra, geometry, and physics in the 21st century},
				},
				book={
					series={Progr. Math.},
					volume={324},
					publisher={Birkh\"{a}user/Springer, Cham},
				},
				date={2017},
				pages={1--29},
				review={\MR{3702382}},
			}
			\bib{MR3959094}{article}{
   author={Donaldson, Simon},
   title={An elliptic boundary value problem for $G_2$-structures},
   language={English, with English and French summaries},
   journal={Ann. Inst. Fourier (Grenoble)},
   volume={68},
   date={2018},
   number={7},
   pages={2783--2809},
   issn={0373-0956},
   review={\MR{3959094}},
   doi={10.5802/aif.3225},
}
			\bib{MR3932259}{article}{
				author={Donaldson, Simon},
				title={Boundary value problems in dimensions 7, 4 and 3 related to
					exceptional holonomy},
				conference={
					title={Geometry and physics. Vol. I},
				},
				book={
					publisher={Oxford Univ. Press, Oxford},
				},
				date={2018},
				pages={115--134},
				review={\MR{3932259}},
			}
			
			
			\bib{MR3966735}{article}{
				author={Donaldson, Simon},
				title={Some recent developments in K\"{a}hler geometry and exceptional
					holonomy},
				conference={
					title={Proceedings of the International Congress of
						Mathematicians---Rio de Janeiro 2018. Vol. I. Plenary lectures},
				},
				book={
					publisher={World Sci. Publ., Hackensack, NJ},
				},
				date={2018},
				pages={425--451},
				review={\MR{3966735}},
			}

						\bib{MR3881202}{article}{
				author={Fine, Joel},
				author={Yao, Chengjian},
				title={Hyper-symplectic 4-manifolds, the $G_2$-Laplacian flow, and
					extension assuming bounded scalar curvature},
				journal={Duke Math. J.},
				volume={167},
				date={2018},
				number={18},
				pages={3533--3589},
				issn={0012-7094},
				review={\MR{3881202}},
				doi={10.1215/00127094-2018-0040},
			}
			\bib{arXiv:2404.15016}{article}{
				author={Fine, Joel},
				author={He, Weiyong},
				author={Yao, Chengjian},
				title={Convergence of the hypersymplectic flow on $T^4$ with $T^3$-symmetry},
				journal={arXiv:2404.15016},
				volume={},
				date={},
				number={},
				pages={},
				issn={},
				review={},
				doi={},
			}
\bib{MR4085670}{article}{
   author={Fine, Joel},
   author={Yao, Chengjian},
   title={A report on the hypersymplectic flow},
   journal={Pure Appl. Math. Q.},
   volume={15},
   date={2019},
   number={4},
   pages={1219--1260},
   issn={1558-8599},
   review={\MR{4085670}},
   doi={10.4310/PAMQ.2019.v15.n4.a7},
}
			
			
			
			\bib{MR1863733}{article}{
				author={Hitchin, Nigel},
				title={The geometry of three-forms in six dimensions},
				journal={J. Differential Geom.},
				volume={55},
				date={2000},
				number={3},
				pages={547--576},
				issn={0022-040X},
				review={\MR{1863733}},
			}
			\bib{MR3873112}{article}{
				author={Huang, Hongnian},
				author={Wang, Yuanqi},
				author={Yao, Chengjian},
				title={Cohomogeneity-one $G_2$-Laplacian flow on the 7-torus},
				journal={J. Lond. Math. Soc. (2)},
				volume={98},
				date={2018},
				number={2},
				pages={349--368},
				issn={0024-6107},
				review={\MR{3873112}},
				doi={10.1112/jlms.12137},
			}
			
		
			
			\bib{MR1424428}{article}{
				author={Joyce, Dominic D.},
				title={Compact Riemannian $7$-manifolds with holonomy $G_2$. I, II},
				journal={J. Differential Geom.},
				volume={43},
				date={1996},
				number={2},
				pages={291--328, 329--375},
				issn={0022-040X},
				review={\MR{1424428}},
			}
			\bib{MR1787733}{book}{
				author={Joyce, Dominic D.},
				title={Compact manifolds with special holonomy},
				series={Oxford Mathematical Monographs},
				publisher={Oxford University Press, Oxford},
				date={2000},
				pages={xii+436},
				isbn={0-19-850601-5},
				review={\MR{1787733}},
			}
\bib{MR3992720}{article}{
   author={Krom, Robin S.},
   author={Salamon, Dietmar A.},
   title={The Donaldson geometric flow for symplectic four-manifolds},
   journal={J. Symplectic Geom.},
   volume={17},
   date={2019},
   number={2},
   pages={381--417},
   issn={1527-5256},
   review={\MR{3992720}},
   doi={10.4310/JSG.2019.v17.n2.a3},
}

\bib{MR1842429}{article}{
   author={Otto, Felix},
   title={The geometry of dissipative evolution equations: the porous medium
   equation},
   journal={Comm. Partial Differential Equations},
   volume={26},
   date={2001},
   number={1-2},
   pages={101--174},
   issn={0360-5302},
   review={\MR{1842429}},
   doi={10.1081/PDE-100002243},
}

\bib{MR3098204}{article}{
   author={Salamon, Dietmar},
   title={Uniqueness of symplectic structures},
   journal={Acta Math. Vietnam.},
   volume={38},
   date={2013},
   number={1},
   pages={123--144},
   issn={0251-4184},
   review={\MR{3098204}},
   doi={10.1007/s40306-012-0004-x},
}

\bib{MR2459454}{book}{
   author={Villani, C\'{e}dric},
   title={Optimal transport},
   series={Grundlehren der mathematischen Wissenschaften [Fundamental
   Principles of Mathematical Sciences]},
   volume={338},
   note={Old and new},
   publisher={Springer-Verlag, Berlin},
   date={2009},
   pages={xxii+973},
   isbn={978-3-540-71049-3},
   review={\MR{2459454}},
   doi={10.1007/978-3-540-71050-9},
}

			\bib{MR4732956}{article}{
				author={Xu, Pengfei},
				author={Zheng, Kai},
				title={The space of closed $G_2$-structures. I. Connections},
				journal={Q. J. Math.},
				volume={75},
				date={2024},
				number={1},
				pages={333--390},
				issn={0033-5606},
				review={\MR{4732956}},
				doi={10.1093/qmath/haae004},
			}
			


		\end{biblist}
	\end{bibdiv}
	\end{document}